\documentclass{amsart}

\usepackage{amsmath}
\usepackage{amssymb}
\usepackage{enumerate}
\usepackage{hyperref}

\newtheorem{theorem}{Theorem}[section]
\newtheorem{lemma}[theorem]{Lemma}
\newtheorem{proposition}[theorem]{Proposition}

\theoremstyle{definition}

\theoremstyle{remark}

\newtheorem*{claim}{Claim}

\numberwithin{equation}{section}

\newcommand{\R}{\mathbb{R}}
\newcommand{\C}{\mathbb{C}}
\newcommand{\HQ}{\mathbb{H}}
\newcommand{\SU}{\mathrm{SU}}
\newcommand{\wtSU}{\widetilde{\mathrm{SU}}}
\newcommand{\U}{\mathrm{U}}
\newcommand{\wtU}{\widetilde{\mathrm{U}}}
\newcommand{\SO}{\mathrm{SO}}

\newcommand{\DD}{\mathcal{D}}
\newcommand{\OO}{\mathcal{O}}
\newcommand{\HH}{\mathcal{H}}
\newcommand{\GG}{\mathcal{G}}
\newcommand{\VV}{\mathcal{V}}
\newcommand{\LL}{\mathcal{L}}
\newcommand{\mS}{\mathcal{S}}

\newcommand{\ZZ}{\mathcal{Z}}
\newcommand{\YY}{\mathcal{Y}}
\newcommand{\X}{\mathfrak{X}}
\newcommand{\g}{\mathfrak{g}}
\newcommand{\h}{\mathfrak{h}}
\newcommand{\so}{\mathfrak{so}}
\newcommand{\su}{\mathfrak{su}}
\newcommand{\uni}{\mathfrak{u}}
\newcommand{\sli}{\mathfrak{sl}}
\newcommand{\gl}{\mathfrak{gl}}

\newcommand{\Iso}{\operatorname{Iso}}
\newcommand{\Kill}{\operatorname{Kill}}
\newcommand{\mfKill}{\mathfrak{Kill}}
\newcommand{\Hom}{\operatorname{Hom}}
\newcommand{\rad}{\operatorname{rad}}

\begin{document}

\title{On low-dimensional manifolds with isometric $\wtU(p,q)$-actions}

\author{Gestur \'Olafsson}
\address{Department of Mathematics, 322 Lockett Hall, Louisiana State
  University, Baton Rouge, LA, 70803, USA}
\email{olafsson@math.lsu.edu}

\author{Raul Quiroga-Barranco}
\address{Centro de Investigaci\'on en Matem\'aticas, Apartado Postal
  402, Guanajuato, Guanajuato, 36250, Mexico}
\email{quiroga@cimat.mx}

\thanks{The research of G.~\'Olafsson was supported by NSF grant DMS-1101337. The research of R.~Quiroga-Barranco was supported by a Conacyt grant and by SNI}

\subjclass{57S20, 53C50, 53C24}
\keywords{Pseudo-Riemannian manifolds, simple Lie groups, rigidity results}

\maketitle

\begin{abstract}
    Denote by $\wtU(p,q)$ the universal covering group of $\U(p,q)$, the linear group of isometries of the pseudo-Hermitian space $\C^{p,q}$ of signature $p,q$. Let $M$ be a connected analytic complete pseudo-Riemannian manifold that admits an isometric $\wtU(p,q)$-action and that satisfies $\dim M \leq n(n+2)$ where $n = p+q$. We prove that if the action of $\wtSU(p,q)$ (the connected derived group of $\wtU(p,q)$) has a dense orbit and the center of $\wtU(p,q)$ acts non-trivially, then $M$ is an isometric quotient of manifolds involving simple Lie groups with bi-invariant metrics. Furthermore, the $\wtU(p,q)$-action is lifted to $\widetilde{M}$ to natural actions on the groups involved. As a particular case, we prove that when $\widetilde{M}$ is not a pseudo-Riemannian product, then its geometry and $\wtU(p,q)$-action are obtained from one of the symmetric pairs $(\su(p,q+1), \uni(p,q))$ or $(\su(p+1,q), \uni(p,q))$.
\end{abstract}

\section{Introduction}
Let $G$ be a connected non-compact simple Lie group. A fundamental dynamical problem is the study of $G$-actions on manifolds. Some of the best known and interesting $G$-actions are obtained from Lie group homomorphisms $G \hookrightarrow S$ where $S$ is some connected group that admits a lattice $\Gamma$. In such situation, $G$ acts analytically by left translations on the manifold $S / \Gamma$ preserving a finite volume. If $S$ is semisimple, then the Killing form of the Lie algebra of $S$ defines a pseudo-Riemannian metric on $S/\Gamma$ which is $G$-invariant. Furthermore, if $\Gamma$ is an irreducible lattice, then the $G$-action is ergodic thus implying that almost every orbit is dense (see \cite{ZimmerSemisimple}). This provides a large family of $G$-spaces with complicated dynamics. A variation of these examples is given by considering a compact subgroup $K \subset S$ that centralizes $G$ and taking the double coset space $K \backslash S / \Gamma$ on which $G$ still acts by left translations. Zimmer proposed a program in \cite{ZimmerProgram} to study finite volume preserving ergodic $G$-actions on manifolds and even Borel spaces. It is generally expected that every such $G$-action is essentially given by a double coset space.

In the development of Zimmer's program, it has been very useful to consider a $G$-invariant geometric structure on the manifold being acted upon. Some of the most important tools were developed by Gromov in \cite{Gromov} to obtain interesting obstructions for manifolds to admit $G$-actions in the presence of a suitable geometry. We refer to our bibliography for other similar or related works. Nevertheless, there are very few results concluding that a manifold admitting a $G$-action preserving a finite volume is in fact a double coset space as above, even when there is a $G$-invariant geometry.

The main contribution of this work is to provide a very explicit description of spaces acted upon by a pseudo-unitary group preserving a pseudo-Riemannian metric for some low dimensional cases. It is also important to note that our results provide, to the best of our knowledge, some of the first global rigidity conclusions around Zimmer's program for isometric actions of a reductive group with non-compact center and not just a non-compact simple Lie group. Another example of results for actions of not necessarily simple Lie groups can be found in \cite{BFM}. The latter work deals with actions preserving parabolic geometries. On the other hand, we consider metric preserving actions which turn out to be based on non-parabolic models.

Recall that $\U(p,q)$ denotes the group of matrices that define the linear isometries of the pseudo-Hermitian space $\C^{p,q}$ of signature $p,q$, and that $\SU(p,q)$ denotes the subgroup of $\U(p,q)$ consisting of the matrices with determinant $1$. We will follow the notation where the universal covering space of $N$ is denoted by $\widetilde{N}$. In particular $\wtSU(p,q) \subset \wtU(p,q)$ denote the universal covering groups of $\SU(p,q)$ and $\U(p,q)$, respectively.

We study an isometric $\wtU(p,q)$-action on a finite volume analytic complete pseudo-Riemannian manifold $M$ so that it has a dense $\wtSU(p,q)$-orbit and so that the connected component $Z(\wtU(p,q))_0$ of the center of $\wtU(p,q)$ acts non-trivially. Furthermore, we will consider the lower dimensional case. More precisely, we will assume that the following inequality holds:
\[
    \dim M \leq \dim \U(p,q) + 2n = n(n+2) \quad \text{where } n = p+q.
\]
Observe that our bound is precisely the dimension of the groups $\SU(p,q+1)$ and $\SU(p+1,q)$.

The following main results prove that a manifold $M$ admitting a $\wtU(p,q)$-action as above can always be constructed in terms of the Lie group $\wtU(p,q)$ and their higher dimension siblings.

First, we obtain the description of $\widetilde{M}$ up to diffeomorphism.

\begin{theorem}[Diffeomorphism type]
    \label{thm:difftype}
    Let $M$ be a connected analytic complete pseudo-Riemannian manifold with finite volume that admits an analytic and isometric $\wtU(p,q)$-action. Suppose that the following are satisfied.
    \begin{itemize}
        \item The $\wtSU(p,q)$-action on $M$ has a dense orbit and the $Z(\wtU(p,q))_0$-action is non-trivial.
        \item $\dim M \leq n(n+2)$, where $p,q \geq 1$, $n = p+q \geq 3$ and $(p,q) \not= (2,2)$
    \end{itemize}
    Then, there exist a diffeomorphism $\varphi : N \rightarrow \widetilde{M}$ where the manifold $N$ is given by one of the following possibilities.
    \begin{enumerate}
        \item $N = \wtSU(p,q) \times N_1$ for some simply connected manifold $N_1$.
        \item $N$ is either $\wtSU(p,q+1)$ or $\wtSU(p+1,q)$.
        \item $N = K\backslash \widetilde{S} \times \R$ where $\widetilde{S}$ is either $\wtSU(p,q+1)$ or $\wtSU(p+1,q)$ and $K$ is a closed subgroup of $\widetilde{S}$ isomorphic to $\R$.
    \end{enumerate}
\end{theorem}

We recall that on the Lie group $\wtSU(p,q)$ there is a unique (up to a constant) bi-invariant pseudo-Riemannian metric and that it is defined by the Killing form of its Lie algebra. We now describe the geometry that $M$ must admit.

\begin{theorem}[Metric type]
    \label{thm:metrictype}
    With the hypotheses and notation of Theorem~\ref{thm:difftype}, we can choose $\varphi$ and replace the pseudo-Riemannian metric on $M$ by one which is still $\wtU(p,q)$-invariant so that $\varphi$ is an isometry for the following pseudo-Riemannian metric on $N$ corresponding to the cases stated in Theorem~\ref{thm:difftype}.
    \begin{enumerate}
        \item For $N = \wtSU(p,q) \times N_1$ the metric is a product metric where $\wtSU(p,q)$ carries a bi-invariant pseudo-Riemannian metric.
        \item For $N$ either $\wtSU(p,q+1)$ or $\wtSU(p+1,q)$ the metric is a bi-invariant pseudo-Riemannian metric.
        \item For $N = K\backslash \widetilde{S} \times \R$ the metric is the product metric obtained from bi-invariant pseudo-Riemannian metrics on $\widetilde{S}$ and $\R$.
    \end{enumerate}
\end{theorem}

Finally we prove that it is also possible to describe the $\wtU(p,q)$-action with respect to the diffeomorphism $\varphi$ from Theorem~\ref{thm:difftype}. Note that the canonical embeddings of $\uni(p,q)$ into either $\su(p,q+1)$ or $\su(p+1,q)$ induce corresponding homomorphisms of the Lie group $\wtU(p,q)$ into either $\wtSU(p,q+1)$ or $\wtSU(p+1,q)$, respectively. In what follows we will refer to either of these homomorphisms  as a canonical symmetric pair embedding of $\wtU(p,q)$ into the corresponding Lie group.

\begin{theorem}[Action type]
    \label{thm:actiontype}
    With the hypotheses and notation of Theorem~\ref{thm:difftype}, we can choose $\varphi$ so that it is $\wtU(p,q)$-equivariant for the following $\wtU(p,q)$-actions on $N$ corresponding to the cases stated in Theorem~\ref{thm:difftype}.
    \begin{enumerate}
        \item For $N = \wtSU(p,q) \times N_1$, the $\wtU(p,q)$-action is the product action of the left translation $\wtSU(p,q)$-action on the first factor and some $Z(\wtU(p,q))_0$-action on $N_1$.
        \item For $N$ either $\wtSU(p,q+1)$ or $\wtSU(p+1,q)$, the $\wtU(p,q)$-action is the left translation action obtained from the canonical symmetric pair embedding of $\wtU(p,q)$ into $\wtSU(p,q+1)$ or $\wtSU(p+1,q)$, respectively.
        \item For $N = K\backslash \widetilde{S} \times \R$, the $\wtU(p,q)$-action is a product action of a $\wtSU(p,q)$-action on $K \backslash \widetilde{S}$ and a $Z(\wtU(p,q))_0$-action on $\R$. The former is an isometric action of $\wtSU(p,q)$ on $K \backslash\widetilde{S}$ for some metric obtained from a bi-invariant metric on $\widetilde{S}$. The latter is a translation action defined by some isomorphism $Z(\wtU(p,q))_0 \simeq \R$.
    \end{enumerate}
\end{theorem}

As a consequence, we also have the following result that singles out a particularly interesting conclusion: that the $\wtU(p,q)$-space $M$ is in fact obtained from a symmetric pair, which is (locally) equivalent to either $(\SU(p,q+1), \U(p,q))$ or $(SU(p+1,q), \U(p,q))$. We recall that a pseudo-Riemannian manifold is weakly irreducible if the holonomy representation on the tangent space at some point has no proper non-degenerate invariant subspace. In particular, a weakly irreducible pseudo-Riemannian manifold cannot have a non-trivial product as universal covering space.

\begin{theorem}[Weakly irreducible case]
    \label{thm:irredcase}
    With the hypotheses and notation of Theorem~\ref{thm:difftype}, we further assume that $M$ is weakly irreducible. Then, for $N$ either $\wtSU(p,q+1)$ or $\wtSU(p+1,q)$ and for the left translation $\wtU(p,q)$-action on $N$ obtained from the canonical symmetric pair embedding $\wtU(p,q) \rightarrow N$, there exists a $\wtU(p,q)$-equivariant diffeomorphism $\varphi : N \rightarrow \widetilde{M}$. Furthermore, the metric on $M$ can be replaced by one which is still $\wtU(p,q)$-invariant and so that $\varphi$ is an isometry where $N$ carries a bi-invariant metric.
\end{theorem}

As for the nature of the fundamental group $\pi_1(M)$ we can state the following consequence of the previous results. The description of $\pi_1(M)$ is particularly explicit in case (2) of the above theorems.

\begin{theorem}[Fundamental group]
    \label{thm:pi1(M)}
    With the hypotheses and notation of Theorem~\ref{thm:difftype}, the fundamental group $\pi_1(M)$ is isomorphic to a discrete subgroup in the group of isometries $\Iso(N)$ where $N$ is one of the pseudo-Riemannian manifolds described in the conclusions of Theorem~\ref{thm:metrictype}. If case (2) in Theorem~\ref{thm:metrictype} holds, then $\pi_1(M)$ has a finite index subgroup isomorphic to a discrete subgroup of $C \times \widetilde{S}$ where $\widetilde{S}$ is either $\wtSU(p,q+1)$ or $\wtSU(p+1,q)$ and $C$ is the centralizer of $\wtSU(p,q)$ in $\widetilde{S}$ with respect to the canonical symmetric pair embedding.
\end{theorem}

Some remarks are in order with respect to our assumptions. In \cite{BF} it was shown the existence of examples of $G$-spaces $M$ with an invariant finite ergodic volume that are not of the double coset type. However, these examples are constructed from double coset spaces by blowing up orbits that are then glued together (see also \cite{KatokLewis}). It was proved in \cite{BF} that such examples admit a $G$-invariant connection but only in the complement of the set where the surgery was applied: it does not seem possible to extend the connection everywhere. Such counterexamples to the conjecture in Zimmer's program thus have a sort of incompleteness property in a geometric sense.

On the other hand, smooth non-analytic geometric structures have a degree of freedom that may produce a similar incompleteness behavior. This is already observed in the Riemannian case considered in \cite{Nomizu}, where it is proved that only for analytic Riemannian manifolds there is a ``uniform'' (i.e.~not depending on the point) way to extend infinitesimal Killing vector fields to local ones. This sort of ``non-uniform'' behavior of the infinitesimal symmetries is also observed in the examples from \cite{BF}.

Hence, it is natural to look among complete analytic geometries for a positive answer and explicit realization of Zimmer's program. This is the reason for us to consider the study of actions on complete analytic pseudo-Riemannian manifolds.

The organization of the work is the following. In Sections~\ref{sec:Kill-G-actions} and \ref{sec:centralizerHH} we present some of the basic tools on Killing vector fields used in the following sections. The results in these sections apply to a very general connected Lie group $G$. It is even allowed of $G$ to have non-trivial center. Section~\ref{sec:Kill-Upq-actions} specializes to the case of $\wtU(p,q)$-actions developing useful properties for the centralizer of the $\wtU(p,q)$-action in the Lie algebra of Killing vector fields of $\widetilde{M}$. Section~\ref{sec:proofmainthms} completes the proofs of the results stated in this Introduction. Finally, an Appendix provides some representation theory results needed in the other sections.

\section{Killing fields of isometric actions}
\label{sec:Kill-G-actions}
For a pseudo-Riemannian manifold $N$ we denote by $\Kill(N)$ the Lie algebra of globally defined Killing vector fields on $N$. Also, we will denote by $\Kill_0(N,x)$ the Lie subalgebra of $\Kill(N)$ consisting of those vector fields that vanish at $x$.

We start by stating the following simple result which provides a representation of $\Kill_0(N,x)$ on $T_xN$.

\begin{lemma}
    \label{lem:lambda}
    Let $N$ be a pseudo-Riemannian manifold and $x \in N$. Then, the map $\lambda_x : \Kill_0(N,x) \rightarrow \so(T_x N)$ given by $\lambda_x(Z)(v) = [Z,V]_x$, where $V$ is any vector field such that $V_x = v$, is a well defined homomorphism of Lie algebras.
\end{lemma}

In the rest of this work $G$ will denote a simply connected Lie group with reductive Lie algebra. We will denote by $[G,G]$ and $Z(G)_0$ the connected subgroups whose Lie algebras are $\DD(\g) = [\g,\g]$ and $Z(\g)$, respectively, and we will assume that $[G,G]$ is non-compact and simple. We also assume that $G$ acts isometrically on a connected finite volume pseudo-Riemannian manifold $M$ so that there is a dense $[G,G]$-orbit and so that for the $Z(G)_0$-action there is at least one point in $M$ with discrete stabilizer; this last condition is equivalent to requiring that the $Z(G)_0$-action is locally free on some non-empty subset of $M$. We also consider the $G$-action on $\widetilde{M}$ lifted from the $G$-action on $M$. Finally, we will assume that $M$ and the $G$-action on $M$ are both analytic. Hence, the $Z(G)_0$-action is locally free on an open dense conull subset of $M$, which is in fact the complement of a proper analytic subset of $M$.

It is well known that the $[G,G]$-action on $M$ is everywhere locally free (see \cite{Szaro,Zeghib}). Hence, the set of $[G,G]$-orbits defines a foliation $\OO$ on $M$, whose tangent bundle will be denoted by $T\OO$. In particular, there is a natural isomorphism of vector bundles given by the expression
\begin{align*}
    M \times \DD(\g) &\rightarrow T\OO \\
    (x, X) &\mapsto X^*_x,
\end{align*}
For $X \in \g$ we denote by $X^*$ the vector field on $M$ whose local flow is $\exp(tX)$. Also, we will denote by $T\OO^\perp$ the bundle whose fibers are the subspaces orthogonal to the fibers of $T\OO$. In what follows, we will use the same symbols $\OO$, $T\OO$ and $T\OO^\perp$ for the corresponding objects on $\widetilde{M}$.

The following result is fundamental for our work. Its proof is based in Proposition~2.3 from
\cite{QuirogaCh} of which it is an extension (see also \cite{Gromov,ZimmerAut,OQ-SO}). Note that for $H$ a connected subgroup of $G$ we denote by $Z_{\X(\widetilde{M})}(H)$ the space of vector fields on $\widetilde{M}$ that centralize the $H$-action. In particular, the main addition of the next result with respect to Proposition~2.3 from \cite{QuirogaCh} is to ensure that the images of the homomorphisms $\rho_x$ below centralize the $Z(G)_0$-action.

For simplicity, from now on we will denote
\[
    \mfKill(\widetilde{M},Z(G)_0) = \Kill(\widetilde{M}) \cap Z_{\X(\widetilde{M})}(Z(G)_0),
\]
which is clearly a finite dimensional Lie algebra.

\begin{proposition}
    \label{prop:Dg(x)}
    Let $G$ and $M$ be as above. Then, there is a dense conull subset $A \subset \widetilde{M}$ such that for every $x \in A$ the $Z(G)_0$-action has discrete stabilizer at $x$ and the following properties are satisfied.
    \begin{enumerate}
    \item There is a homomorphism $\rho_x : \DD(\g) \rightarrow \mfKill(\widetilde{M},Z(G)_0)$ which is an isomorphism onto its image $\rho_x(\DD(\g))$.
    \item Every element of $\rho_x(\DD(\g))$ vanishes at $x$.
    \item For every $X \in \DD(\g)$ and $Y \in \g$  we have:
        $$
            [\rho_x(X),Y^*] = [X,Y]^* = -[X^*,Y^*].
        $$
        In particular, the elements in $\rho_x(\DD(\g))$ and their corresponding local flows preserve both $\OO$ and $T\OO^\perp$, and we also have
        \[
            [\rho_x(X),Y^*] = 0
        \]
        for every $X \in \DD(\g)$ and $Y \in Z(\g)$.
    \item The homomorphism of Lie algebras $\lambda_x\circ\rho_x : \DD(\g) \rightarrow \so(T_x \widetilde{M})$ induces a $\DD(\g)$-module structure on $T_x \widetilde{M}$ for which the subspaces $T_x \OO$ and $T_x \OO^\perp$ are $\DD(\g)$-submodules.
    \end{enumerate}
\end{proposition}
\begin{proof}
    The statement of Proposition~2.3 from \cite{QuirogaCh} applied to the Lie group $[G,G]$ is precisely our statement except for the claims that for every $X \in \DD(\g)$ we have $\rho_x(X) \in Z_{\X(\widetilde{M})}(Z(G))_0$ as well as $[\rho_x(X),Y^*] = 0$ for all $Y \in Z(\g)$, which are in fact equivalent conditions. We will explain how to modify the arguments used in \cite{QuirogaCh} to achieve this additional property. For this we note that the existence of the required vector fields is first established infinitesimally, in the sense of jets, and this is followed by local and global extensions. We will use the notation from \cite{QuirogaCh} and we refer to this work and \cite{CQGromov} for the basic facts on jets and geometric structures. We will also use the results from \cite{CQParallelisms}. In what follows we will denote by the same symbol a geometric structure when lifted from $M$ to $\widetilde{M}$.

    Let $W$ be the subset of $M$ where the $Z(G)_0$-action is locally free. Hence $W$ is an open dense conull subset of $M$. In particular, the $Z(G)_0$-action defines a vector subbundle of $TW$ and, after fixing a base of $Z(\g)$, a parallelism of such subbundle. The latter in turn defines an order $1$ geometric structure $\omega$ on $W$. Let $\widetilde{W}$ be the inverse image of $W$ in $\widetilde{M}$ with respect to the covering map $\widetilde{M} \rightarrow M$. Then, the Lie algebra of global Killing vector fields of $\omega$ on $\widetilde{W}$ is given precisely by $\Kill(\widetilde{W}, \omega) = Z_{\X(\widetilde{W})}(Z(G)_0)$, and similar identities hold for local and infinitesimal Killing fields. We also observe that since the $G$-action and the $Z(G)_0$-action centralize each other, the group $G$ preserves both $W$ and the geometric structure $\omega$.

    If we denote by $\omega'$ the geometric structure defined by the pseudo-Riemannian metric on $M$, then the map
    \begin{align*}
        \sigma : L(W) &\rightarrow Q_1 \times Q_2\\
            \sigma(\alpha) &= (\omega(\alpha), \omega'(\alpha)),
    \end{align*}
    defines an order $1$ geometric structure on $W$, where $Q_1$ and $Q_2$ are suitably chosen. Note that we have adopted the notation where a geometric structure is an equivariant map defined on a frame bundle into a suitable manifold acted upon by the structure group of the frame bundle. With this definition, it is clear that the Killing fields of $\sigma$ are the Killing fields for both $\omega'$ (and so of the pseudo-Riemannian metric) and $\omega$, and this holds for infinitesimal, local and global symmetries; in other words, we have
    \[
        \Kill(\widetilde{W}, \sigma) = \Kill(\widetilde{W}) \cap \Kill(\widetilde{W}, \omega)
            = \Kill(\widetilde{W}) \cap Z_{\X(\widetilde{W})}(Z(G)_0)
            = \mfKill(\widetilde{W},Z(G)_0)
    \]
    as well as the corresponding identities for local and infinitesimal Killing fields. In particular, the rigidity of the pseudo-Riemannian metric implies that of $\sigma$. We observe that $\sigma$ is analytic and of algebraic type since both $\omega$ and $\omega'$ have such properties, and $\sigma$ is unimodular because of the volume form associated to the pseudo-Riemannian metric. We also note that $\sigma$ is $G$-invariant as a consequence of the $G$-invariance of the pseudo-Riemannian metric and of $\omega$.

    The first part of the proof of Proposition~2.3 from \cite{QuirogaCh} applied to the connected non-compact simple Lie group $[G,G]$ corresponds to the infinitesimal existence and it is based on the proof of Lemma~9.1 from \cite{CQGromov}. Since the latter holds for a general $[G,G]$-invariant analytic unimodular rigid geometric structure of algebraic type, we conclude from \cite{QuirogaCh,CQGromov} the existence of a dense conull subset $A \subset \widetilde{W}$ such that for every $x \in A$ and for every $k \geq 1$ there exists an injective homomorphism
    \[
        \rho_x^k : \DD(\g) \rightarrow \Kill_0^k(\widetilde{W}, \sigma, x),
    \]
    such that
    \[
        [\rho_x^k(X), j^{k-1}_x(Y^*)] = j^{k-1}_x([X,Y]^*)
    \]
    for every $X,Y \in \DD(\g)$, where $\Kill_0^k(\widetilde{W}, \sigma, x)$ denotes the vector space of infinitesimal Killing vector fields of order $k$ at $x$ that vanish at $x$. Note that the brackets in the left of the last identity are obtained by taking the jet of order $k-1$ of the brackets of representative vector fields; this is well defined by the remarks found in \cite{CQGromov}.

    For the local extension argument, by Theorem~3 from \cite{CQParallelisms} for every $x \in \widetilde{W}$ there exists $k(x) \geq 1$ such that the map
    \begin{align*}
        J^k_x : \Kill_0^{loc}(\widetilde{W}, \sigma, x) &\rightarrow \Kill_0^k(\widetilde{W}, \sigma, x)  \\
            X &\mapsto j^k_x(X),
    \end{align*}
    is an isomorphism for every $k \geq k(x)$. Here, $\Kill_0^{loc}(\widetilde{W}, \sigma, x)$ denotes the space of Killing vector fields of $\sigma$ defined in a neighborhood of $x$ that vanish at $x$. We conclude that for $x \in A$ and a fixed $k \geq k(x)$ the map
    \[
        \rho^{loc}_x = (J^k_x)^{-1} \circ \rho^k_x :
        \DD(\g) \rightarrow \Kill_0^{loc}(\widetilde{W}, \sigma, x)
    \]
    is an injective Lie algebra homomorphism such that
    \[
        j^{k-1}_x([\rho^{loc}_x(X), Y^*]) = j^{k-1}_x([X,Y]^*),
    \]
    for every $X, Y \in \DD(\g)$. If we choose $k \geq 2$, then the $1$-rigidity and the analyticity of the pseudo-Riemannian metric imply that
    \[
        [\rho^{loc}_x(X), Y^*] = [X,Y]^*
    \]
    in a neighborhood of $x$ for every $X, Y \in \DD(\g)$.

    Finally, for the global extension, we recall that the results from \cite{Amores} imply that the restriction map
    \[
        J^{loc}_x : \Kill_0(\widetilde{M}, x) \rightarrow \Kill^{loc}_0(\widetilde{M}, x)
    \]
    is an isomorphism for every $x \in \widetilde{M}$. This induces for every $x \in A$ an injective homomorphism of Lie algebras
    \[
        \rho_x = (J^{loc}_x)^{-1} \circ \rho_x^{loc} : \DD(\g) \rightarrow \Kill_0(\widetilde{M}, x)
    \]
    that satisfies (3) from our theorem's statement for $X, Y \in \DD(\g)$ as well as the following property.
    \begin{itemize}
        \item For every $X \in \DD(\g)$, the restriction of $\rho_x(X)$ to a neighborhood of $x$ belongs to $\Kill^{loc}(\widetilde{W}, \omega, x)$.
    \end{itemize}
    In other words, for every $X \in \DD(\g)$ we have
    \[
        [\rho_x(X), Y^*] = 0
    \]
    for every $Y \in Z(\g)$ in a neighborhood of $x$. By the analyticity of the objects involved and the connectedness of $\widetilde{M}$ such vanishing holds everywhere thus proving (3) in full generality. It also implies that we have $\rho_x(X) \in Z_{\X(\widetilde{M})}(Z(G)_0)$ for every $X \in \DD(\g)$, which completes the proof of (1).

    We note that $A$ is a dense conull subset in $\widetilde{M}$ since $\widetilde{W}$ is open dense conull in $\widetilde{M}$. The rest of the claims are now easily obtained.
\end{proof}

It remains to determine whether or not the $Z(G)_0$-action on $M$ is locally free everywhere. The next result provides a simple criterion that will be useful latter on. In the rest of this work, we will denote for $x \in M$
\[
    \ZZ_x = \{ Z^*_x : Z \in Z(\g) \},
\]
which is the tangent space at $x$ to the $Z(G)_0$-orbit of $x$. When considering the $Z(G)_0$-action on $\widetilde{M}$ we will use the same notation.

\begin{proposition}
    \label{prop:Zx}
    Let $G$ and $M$ be as above. Assume that at some point of $M$ the $Z(G)_0$-action is locally free with some non-degenerate orbit as a submanifold of the pseudo-Riemannian manifold $M$. Then, the $Z(G)_0$-action on $M$ is locally free everywhere with non-degenerate orbits and the same holds for the $Z(G)_0$-action on $\widetilde{M}$ as well. In particular, the subset
    \[
        \ZZ = \bigcup_{x \in \widetilde{M}} \ZZ_x
    \]
    is the tangent bundle to the foliation by $Z(G)_0$-orbits and so it is an analytic subbundle of $T\widetilde{M}$.
\end{proposition}
\begin{proof}
    Let $Z_1, \dots, Z_k$ be a basis for $Z(\g)$, and let $Z_1^*, \dots, Z_k^*$ be the corresponding Killing fields on $M$. Consider the analytic function
    \begin{align*}
        F : M &\rightarrow \R \\
        F(x) &= \det\left(\left(h_x(Z_r^*(x),Z_s^*(x))\right)_{r,s=1}^k\right),
    \end{align*}
    where $h$ is the pseudo-Riemannian metric of $M$. Since $h$ and every vector field $Z_j^*$ are $[G,G]$-invariant, the function $F$ is $[G,G]$-invariant as well. We conclude that $F$ is constant since $M$ has a dense $[G,G]$-orbit. On the other hand, by hypothesis $F(x_0) \not= 0$ at some point $x_0$ and so $F(x) = F(x_0) \not= 0$ at every $x \in M$. In particular, the $Z(G)_0$-orbits are non-degenerate at every point. It also implies that the vector fields $Z_1^*, \dots, Z_k^*$ are linearly independent everywhere.
\end{proof}

\section{The centralizer of isometric actions}
\label{sec:centralizerHH}
We continue considering $G$ and $M$ satisfying the conditions given in Section~\ref{sec:Kill-G-actions}. Also, from now on $\HH$ will denote the centralizer in $\mfKill(\widetilde{M}, Z(G)_0)$ of the $[G,G]$-action on $\widetilde{M}$. Note that the inclusion $\HH \subset \mfKill(\widetilde{M}, Z(G)_0)$ implies that $\HH$ centralizes the $Z(G)_0$-action as well. In this section we present some basic structure theory for this centralizer $\HH$.

The following local homogeneity result is well known and it is a particular case of Gromov's open dense orbit theorem. For its proof we refer to \cite{Gromov} and \cite{ZimmerAut} (see also \cite{An,Melnick}).
\begin{proposition}
    \label{prop:GromovOpenDense}
    For $G$ and $M$ satisfying the above conditions, there is an open dense conull subset $U \subset \widetilde{M}$ such that for every $x \in U$ the evaluation map
    \begin{align*}
        ev_x : \HH &\rightarrow T_x \widetilde{M} \\
            Z &\mapsto Z_x
    \end{align*}
    is surjective.
\end{proposition}
\begin{proof}
    To conclude this result we observe that, with the notation of the proof of Proposition~\ref{prop:Dg(x)}, we have $\mfKill(\widetilde{W}, Z(G)_0) = \Kill(\widetilde{W},\sigma)$. Hence, it is enough to apply the results from \cite{Gromov} or \cite{ZimmerAut} to the geometric structure $\sigma$ preserved by $[G,G]$ to obtain an open dense conull subset $U \subset \widetilde{W}$ that satisfies the conclusions. For this we also note that $U$ is open dense conull in $\widetilde{M}$ since $\widetilde{W}$ is so in $\widetilde{M}$.
\end{proof}

The following result is similar to Lemma~1.7 from \cite{OQ-SO}.

\begin{lemma}
    \label{lem:G(x)}
    For $G$ and $M$ as above, let $A$ be given as in Proposition~\ref{prop:Dg(x)}. Then, for every $x\in A$ and for $\rho_x$ as in Proposition~\ref{prop:Dg(x)}, the map
    \begin{align*}
        \widehat{\rho}_x : \g &\rightarrow \mfKill(\widetilde{M}, Z(G)_0) \\
        \widehat{\rho}_x(X + Y) &= \rho_x(X) + X^* + Y^*,
    \end{align*}
    where $X \in \DD(\g)$ and $Y \in Z(\g)$, is an injective homomorphism of Lie algebras whose image $\GG(x)$ lies in $\HH$. In particular, $\GG(x) \simeq \g$ as Lie algebras, Furthermore, the following properties hold.
    \begin{enumerate}
        \item The Lie brackets of $\HH$ turn it into a $\GG(x)$-module.
        \item If we denote by $\ZZ\GG(x)$ the center of $\GG(x)$, then we have
            \[
                \ZZ\GG(x) = \widehat{\rho}_x(Z(\g)) = \{ Z^* : Z \in Z(\g) \}
            \]
            and $[\HH, \ZZ\GG(x)] = 0$. In particular, for every $Z \in Z(\g)$ the subspace $\R Z^*$ is a trivial $\GG(x)$-submodule of $\HH$.
    \end{enumerate}
\end{lemma}
\begin{proof}
    The map $\widehat{\rho}_x$ is clearly linear with values lying in $\mfKill(\widetilde{M}, Z(G)_0)$ since both $\rho_x(X)$ and $Y^*$ preserve the metric on $\widetilde{M}$ and commute with the $Z(G)_0$-action, for $X \in \DD(\g)$ and $Y \in \g$. To prove that it is a homomorphism of Lie algebras, given $X_1, X_2 \in \DD(\g)$ and $Y_1, Y_2 \in Z(\g)$ we observe that from Proposition~\ref{prop:Dg(x)} we obtain
    \begin{align*}
        [\widehat{\rho}_x(X_1 + Y_1), \widehat{\rho}_x(X_2 + Y_2)]
            =& [\rho_x(X_1) + X_1^* + Y_1^*, \rho_x(X_2) + X_2^* + Y_2^*] \\
            =& [\rho_x(X_1) + X_1^*, \rho_x(X_2) + X_2^*] \\
            =& [\rho_x(X_1),\rho_x(X_2)] + [\rho_x(X_1), X_2^*] \\
                &+ [X_1^*, \rho_x(X_2)] + [X_1^*, X_2^*] \\
            =& \rho_x([X_1, X_2]) + [X_1, X_2]^* \\
                &+ [X_1, X_2]^* - [X_1, X_2]^* \\
            =& \rho_x([X_1, X_2]) + [X_1, X_2]^* \\
            =& \widehat{\rho}_x([X_1, X_2]) \\
            =& \widehat{\rho}_x([X_1 + Y_1, X_2 + Y_2]).
    \end{align*}
    If we assume that $\widehat{\rho}_x(X) = 0$ for some $X \in \DD(\g)$, then evaluation at $x$ yields $X^*_x = 0$. Since the $[G, G]$-action is locally free this implies that $X = 0$. In particular, the Lie algebra $\widehat{\rho}_x(\DD(\g))$ is isomorphic to $\DD(\g)$ and hence simple. Now suppose that for some $X \in \DD(\g)$ and $Y \in Z(\g)$ we have
    \[
        \widehat{\rho}_x(X) + Y^* = 0.
    \]
    Then, for every $Z \in \DD(\g)$ we have
    \[
        [\widehat{\rho}_x(X), \widehat{\rho}_x(Z)] = -[Y^*, \rho_x(Z) + Z^*] = [Z, Y]^* + [Y, Z]^* = 0,
    \]
    which shows that $\widehat{\rho}_x(X)$ lies in the center of $\widehat{\rho}_x(\DD(\g))$ thus implying that $X = 0$. This yields $Y^* = 0$ and so $Y = 0$ since the $Z(G)_0$-action is locally free on a non-empty subset. Hence, the map $\widehat{\rho}_x$ is injective.

    To prove that $\GG(x) \subset \HH$, we observe that for $X,Y \in \DD(\g)$ and $Z \in Z(\g)$ we have
    \begin{align*}
        [\widehat{\rho}_x(X + Z), Y^*]
            &= [\rho_x(X) + X^* + Z^*, Y^*] \\
            &= [X, Y]^* + [X^*, Y^*] + [Z^*, Y^*] \\
            &= [Y, Z]^* = 0.
    \end{align*}

    On the other hand, by the properties of $\widehat{\rho}_x$ proved so far we clearly have
    \[
        \ZZ\GG(x) = \widehat{\rho}_x(Z(\g)) = \{ Z^* : Z \in Z(\g) \}.
    \]
    Since $\HH \subset Z_{\X(\widetilde{M})}(Z(G)_0)$, we have that $[\HH, Z^*] = 0$ for every $Z \in Z(\g)$. This completes the proof of (2).
\end{proof}

In the rest of this work we will denote by $\DD\GG(x)$ the derived subalgebra and by $\ZZ\GG(x)$ the center of $\GG(x)$, respectively.

Next we show that Proposition~\ref{prop:GromovOpenDense} allows us to define a $\GG(x)$-module structure on $T_x\widetilde{M}$. Furthermore, we prove that the natural evaluation map intertwines the $\GG(x)$-module structures on $\HH$ and $T_x\widetilde{M}$. Note that by Lemma~\ref{lem:G(x)} the map $\widehat{\rho}_x$ provides a particular realization of the isomorphism $\GG(x) \simeq \g$. The latter allows us to describe the isomorphism types of $\GG(x)$-modules in terms of known $\g$-modules. We will make use of this in the rest of the work.

\begin{lemma}
    \label{lem:TxM-module-structure}
    For $G$ and $M$ as above, let $A$ and $U$ be as in Propositions~\ref{prop:Dg(x)} and \ref{prop:GromovOpenDense}, respectively. Fix some point $x \in A \cap U$. Then, the following properties hold.
    \begin{enumerate}
        \item The map $\lambda_x : \GG(x) \rightarrow \so(T_x \widetilde{M})$ given by $\lambda_x(Z)(v) = [Z,V]_x$, where $V \in \HH$ is such that $V_x = v$, is a well defined homomorphism of Lie algebras.
        \item The evaluation map $ev_x : \HH \rightarrow T_x \widetilde{M}$ is a homomorphism of $\GG(x)$-modules, and it satisfies $ev_x(\DD\GG(x)) = T_x \OO$. In particular, $T_x\OO$ is a $\GG(x)$-module isomorphic to the $\g$-module $\DD(\g)$.
        \item The subspace $T_x\OO^\perp$ is a $\GG(x)$-submodule of $T_x\widetilde{M}$.
  \end{enumerate}
\end{lemma}
\begin{proof}
    By the choice of $x$, for every $v \in T_x\widetilde{M}$ there exists $V \in \HH$ such that $V_x = v$. If $Z \in \GG(x)$ is given, then there are some $X \in \DD(\g)$ and $Y \in Z(\g)$ such that $Z = \rho_x(X) + X^* + Y^*$. With these choices we have
    \[
        [Z,V] = [\rho_x(X) + X^* +Y^*, V] = [\rho_x(X),V],
    \]
    where the second identity follows from the fact that $V$ centralizes the $[G,G]$-action as well as $Y^*$ by Lemma~\ref{lem:G(x)}(2). Since $\rho_x(X)$ vanishes at $x$, this  shows that the dependence of $[Z,V]_x$ on $V$ is only on $V_x = v$. In particular, the map $\lambda_x$ given above is well  defined. That $\lambda_x$ is a homomorphism of Lie algebras into $\gl(T_x\widetilde{M})$ follows from the Jacobi identity and the fact that $\HH$ is a $\GG(x)$-module.

    Next, for $h$ the metric of $\widetilde{M}$ and for every $X \in \DD(\g)$, we have
    \[
        h_x([\rho_x(X),V]_x, V'_x) + h_x(V_x, [\rho_x(X),V']_x) = 0,
    \]
    for every pair of vector fields $V, V' \in \HH$. This is a consequence of the fact that $\rho_x(X)$ is a Killing vector field that vanishes at $x$. Hence, for $V, V' \in \HH$, $X \in \DD(\g)$ and $Y \in Z(\g)$, the previous computations show that
    \[
        h_x([\rho_x(X) + X^* + Y^*,V]_x, V'_x) + h_x(V_x, [\rho_x(X) + X^* + Y^*,V']_x) = 0,
    \]
    thus proving that for every $Z \in \GG(x)$ and every $v, v' \in T_x(\widetilde{M})$ we have
    \[
        h_x(\lambda_x(Z)(v), v') + h_x(v, \lambda_x(Z)(v')) = 0.
    \]
    We conclude that $\lambda_x(Z) \in \so(T_x\widetilde{M})$ for every $Z \in \GG(x)$, thus completing the proof of (1).

    On the other hand, from the definitions involved, it is clear that $ev_x$ is homomorphism of $\GG(x)$-modules and that $ev_x(\DD\GG(x)) = T_x\OO$. That $T_x \OO$ is isomorphic to $\DD(\g)$ as $\g$-module is a consequence of the above expressions and of Proposition~\ref{prop:Dg(x)}(3). This yields (2). Finally, that $T_x\OO^\perp$ is a $\GG(x)$-submodule now follows from (1) and (2).
\end{proof}

For $x \in A\cap U$, in the rest of this work we consider $\HH$ and $T_x\widetilde{M}$ endowed with the $\GG(x)$-module structures defined in Lemmas~\ref{lem:G(x)} and \ref{lem:TxM-module-structure}, respectively. We now introduce a Lie subalgebra of $\HH$ that is very useful to study the structure of $\HH$.

\begin{lemma}
    \label{lem:H0(x)}
    For $G$ and $M$ above, let $A$ and $U$ be as in Propositions~\ref{prop:Dg(x)} and \ref{prop:GromovOpenDense}, respectively. Fix some point $x \in A \cap U$. Then, the subspace $\HH_0(x) = \ker(ev_x)$ is a Lie subalgebra of both $\HH$ and $\Kill_0(\widetilde{M},x)$, as well as a $\GG(x)$-submodule of $\HH$. Furthermore, the sum $\GG(x) + \HH_0(x)$ is direct and it is a Lie subalgebra of $\HH$ that contains $\HH_0(x)$ as an ideal. In particular, $\HH$ is a $\GG(x) + \HH_0(x)$-module.
\end{lemma}
\begin{proof}
    In fact, we have $\HH_0(x) = \HH \cap \Kill_0(\widetilde{M},x)$, which implies that it is a Lie subalgebra of $\HH$. That $\HH_0(x)$ is a $\GG(x)$-submodule follows from the fact that $ev_x$ is a homomorphism of $\GG(x)$-modules.

    Next assume that $Z = \rho_x(X) + X^* + Y^* \in \GG(x)$ vanishes at $x$, where $X \in \DD(\g)$ and $Y \in Z(\g)$. Hence, we have $X^*_x + Y^*_x = ev_x(Z) = 0$ which implies that the inclusion
    \[
        ev_x(\R Y^*) = ev_x(\R X^*) \subset ev_x(\DD\GG(x)) = T_x \OO
    \]
    realizes $ev_x(\R Y^*)$ as a $\GG(x)$-submodule of $T_x \OO$. By Lemma~\ref{lem:TxM-module-structure}(2) the latter is a $\GG(x)$-module isomorphic to $\DD(\g)$ as a $\g$-module, and so we have $ev_x(\R Y^*) = ev_x(\R X^*) = 0$. The local freeness at $x$ of the $[G,G]$ and $Z(G)_0$-actions now imply that $X = Y = 0$. Hence, we conclude that $\GG(x) \cap \HH_0(x) = 0$ and the sum $\GG(x) + \HH_0(x)$ is direct.

    The rest of the statement follows directly from the properties proved so far.
\end{proof}

The constructions considered up to this point yield the following module structure over $\GG(x) + \HH_0(x)$ together with some useful properties.

\begin{proposition}
    \label{prop:GH0-module-structure}
    For $G$ and $M$ as above, let $A$ and $U$ be as in Propositions~\ref{prop:Dg(x)} and \ref{prop:GromovOpenDense}, respectively. For a fixed point $x \in A \cap U$, let $\GG(x)$
    and $\HH_0(x)$ be the Lie subalgebras of $\HH$ defined in Lemmas~\ref{lem:G(x)} and \ref{lem:H0(x)}, respectively. Consider the map defined by
    \begin{align*}
        \lambda_x : \GG(x) + \HH_0(x) &\rightarrow \so(T_x\widetilde{M}) \\
        \lambda_x(Z)(v) &= [Z,V]_x,
    \end{align*}
    where for a given $v \in T_x\widetilde{M}$ we choose $V \in \HH$ such that $V_x = v$. Then, the following properties are satisfied.
    \begin{enumerate}
        \item The map $\lambda_x$ is a well defined homomorphism of Lie algebras. In particular, $T_x\widetilde{M}$ is a $\GG(x) + \HH_0(x)$-module.
        \item The evaluation map $ev_x : \HH \rightarrow T_x\widetilde{M}$ is a homomorphism of $\GG(x) + \HH_0(x)$-modules for the module structures on $\HH$ and    $T_x\widetilde{M}$ defined by Lemma~\ref{lem:H0(x)} and (1), respectively.
        \item The subspaces $T_x\OO$ and $T_x\OO^\perp$ are $\GG(x) + \HH_0(x)$-submodules of $T_x\widetilde{M}$.
        \item If $T_x\OO \cap T_x\OO^\perp = 0$, then we have
            \[
                \ZZ_x = ev_x(\ZZ\GG(x)) \subset T_x \OO^\perp.
            \]
    \end{enumerate}
\end{proposition}
\begin{proof}
    Claim (1) is proved with arguments similar to those used in the proof of Lemma~\ref{lem:TxM-module-structure}. Hence, (2) is an immediate consequence of the definition of $\lambda_x$.

    By (1) and Lemma~\ref{lem:TxM-module-structure}(2) to prove (3) it is enough to show that $\HH_0(x)$ leaves invariant $T_x\OO$. For this we observe that, from the previous results we have
    \[
        T_x\OO = ev_x(\DD\GG(x)), \quad [\HH_0(x),\GG(x)] \subset \HH_0(x)
    \]
    and $ev_x$ is a homomorphism of $\HH_0(x)$-modules; these imply that $\lambda_x(\HH_0(x))(T_x\OO) = 0$. In particular, $T_x\OO$ is a trivial $\HH_0(x)$-module.

    Now assume that $T_x\OO \cap T_x \OO^\perp = 0$. We will prove that $ev_x(\ZZ\GG(x)) \subset T_x \OO^\perp$. By Lemma~\ref{lem:G(x)}(2), if we choose an element of $\ZZ\GG(x)$, then we can assume it is of the form $Z^*$ for some $Z \in Z(\g)$. Suppose that $Z^*_x = u + v$ where $u \in T_x \OO$ and $v \in T_x \OO^\perp$. Hence, there exist $X \in \DD(\g)$ and $V \in \HH$ such that $\widehat{\rho}_x(X)_x = u$ and $V_x = v$; the latter can be picked by Proposition~\ref{prop:GromovOpenDense}. In particular, $Z^* - \widehat{\rho}_x(X) - V \in \HH$ and vanishes at $x$, which implies that
    \[
        Z^* = \widehat{\rho}_x(X) + V + W
    \]
    for some $W \in \HH_0(x)$. By Proposition~\ref{prop:Dg(x)} it follows that for every $Y \in \DD(\g)$ we have
    \begin{align*}
        0 &= [\rho_x(Y), Z^*]_x \\
            &= [\rho_x(Y), \rho_x(X)]_x + [\rho_x(Y), X^*]_x +
                [\rho_x(Y), V]_x + [\rho_x(Y), W]_x \\
            &= \rho_x([Y,X])_x + [Y,X]^*_x +
                \lambda_x(\rho_x(Y))(V_x) + \lambda_x(\rho_x(Y))(W_x) \\
            &= \widehat{\rho}_x([Y,X])_x + \lambda_x(\rho_x(Y))(v),
    \end{align*}
    where we have also used the definition of $\widehat{\rho}_x$ and the fact that $W_x = 0$. We now observe that $\widehat{\rho}_x([Y,X])_x \in T_x \OO$ and $\lambda_x(\rho_x(Y))(v) \in T_x \OO^\perp$, and so both are zero. Since $\widehat{\rho}_x$ is injective, we conclude that $[Y,X] = 0$ for every $Y \in \DD(\g)$. As a consequence $X = 0$ and so $Z^* = V + W$, which proves that
    \[
        Z^*_x = ev_x(Z^*) = ev_x(V) = v \in T_x \OO^\perp.
    \]
\end{proof}

Using Proposition~\ref{prop:GH0-module-structure}, we define the homomorphism of Lie algebras
\begin{align*}
    \lambda_x^\perp : \GG(x) + \HH_0(x) &\rightarrow \so(T_x\OO^\perp)\\
    \lambda_x^\perp(Z) &= \lambda_x(Z)|_{T_x\OO^\perp}.
\end{align*}
The following result proves that $\HH_0(x)$ is completely determined by the representation $\lambda^\perp_x$ when we assume that $T_x\OO \cap T_x\OO^\perp = 0$.

\begin{proposition}
    \label{prop:H0(x)-lambda^perp}
    For $G$ and $M$ as above, let $A$ and $U$ be as in Propositions~\ref{prop:Dg(x)} and \ref{prop:GromovOpenDense}, respectively. For a fixed point $x \in A \cap U$ assume that $T_x\OO \cap T_x\OO^\perp = 0$. Then, the homomorphism of Lie algebras
    \[
        \lambda_x^\perp : \HH_0(x) \rightarrow \so(T_x\OO^\perp)
    \]
    is injective. Furthermore, $\lambda_x^\perp(\HH_0(x))$ is a Lie subalgebra and a $\lambda_x^\perp(\GG(x))$-submodule of $\so(T_x\OO^\perp)$.
\end{proposition}
\begin{proof}
    We recall that every Killing field is completely determined by its $1$-jet at $x$. If we fix $Z \in \HH_0(x)$, then $Z_x = 0$, and so $Z$ is completely determined by $[Z,V]_x$ where $V$ varies in a set of vector fields $\mathcal{A}$ such that $ev_x(\mathcal{A}) = T_x \widetilde{M}$. From the above, we already know that $ev_x(\DD\GG(x)) = T_x\OO$ and $[Z, V]_x = 0$ for every $V \in \DD\GG(x)$. Hence, given the condition $T_x\OO \cap T_x\OO^\perp = 0$, we further have that every $Z \in \HH_0(x)$ is completely determined by $[Z,V]_x$ where $V$ varies in a set of vector fields $\mathcal{A}$ such that $ev_x(\mathcal{A}) = T_x\OO^\perp$, which implies the injectivity of $\lambda_x^\perp$ on $\HH_0(x)$.

    The rest of the claims now follow easily using that $\lambda_x^\perp$ is a homomorphism and that $\HH_0(x)$ is an ideal in $\GG(x) + \HH_0(x)$.
\end{proof}

With the above Lie subalgebras of $\HH$ we now provide a first description of the structure of $\HH$.

\begin{proposition}
    \label{prop:HHwithVV(x)}
    For $G$ and $M$ as above, let $A$ and $U$ be as in Propositions~\ref{prop:Dg(x)} and \ref{prop:GromovOpenDense}, respectively. For a fixed point $x \in A \cap U$ assume that $T_x\OO \cap T_x\OO^\perp = 0$. Then there exists a $\GG(x)$-submodule $\VV(x)$ of $\HH$ such that
    \begin{align*}
        \HH &= \GG(x) \oplus \HH_0(x) \oplus \VV(x) \\
        T_x\OO^\perp &= \ZZ_x \oplus ev_x(\VV(x)) = ev_x(\ZZ\GG(x) \oplus \VV(x)).
    \end{align*}
\end{proposition}
\begin{proof}
    By Proposition~\ref{prop:GH0-module-structure}(4) it follows that $ev_x(\ZZ\GG(x)) = \ZZ_x \subset T_x\OO^\perp$, and we also have $ev_x(\DD\GG(x) + \HH_0(x)) = T_x\OO$. Hence, the simplicity of $\DD\GG(x)$ and Proposition~\ref{prop:GH0-module-structure}(2) and (3) imply the existence of a $\DD\GG(x)$-submodule $\VV(x)$ of $\HH$ with the required properties. By Lemma~\ref{lem:G(x)}(2) we have $[\VV(x), \ZZ\GG(x)] \subset [\HH, \ZZ\GG(x)] = 0$ and so $\VV(x)$ is a $\GG(x)$-submodule.
\end{proof}

We will now consider the integrability of the normal bundle $T\OO^\perp$ for the case where $T_x \OO \cap T_x \OO^\perp = 0$ at every point $x$. The next result provides a necessary condition for this to hold. It is a consequence of Lemma~2.7 from \cite{QuirogaCh} (see also Lemma~1.4 from \cite{OQ-SO}).

\begin{lemma}
    \label{lem:TOoplusTOperp}
    Let $G$ and $M$ be as above. If $\dim(M) < 2\dim [G.G]$, then $T\OO$ and $T\OO^\perp$ have non-degenerate fibers with respect to the  metric of $M$. In particular, we have $TM = T\OO \oplus T\OO^\perp$, a sum of analytic vector subbundles.
\end{lemma}

Assume from now on that $T_x \OO \cap T_x \OO^\perp = 0$ at every point $x$, in other words, that we have $TM = T\OO \oplus T\OO^\perp$ as well as the corresponding property for $\widetilde{M}$. Then, there is an analytic map of vector bundles
\[
    \overline{\omega} : T\widetilde{M} \rightarrow T\OO
\]
given by the orthogonal projection onto $T\OO$. We also recall from Section~\ref{sec:Kill-G-actions} that there is an isomorphism given by
\begin{align*}
    \alpha_x : T_x \OO &\rightarrow \DD(\g) \\
        \alpha_x(X^*_x) &\mapsto X
\end{align*}
that varies analytically with respect to $x$. This yields the analytic $\DD(\g)$-valued $1$-form $\omega$ on $\widetilde{M}$ given by
\[
    \omega_x = \alpha_x \circ \overline{\omega}_x
\]
where $x \in \widetilde{M}$. We introduce the analytic $\DD(\g)$-valued $2$-form $\Omega$ given by
\[
    \Omega_x = d\omega_x|_{\wedge^2 T_x\OO^\perp},
\]
for every $x \in \widetilde{M}$. If $X, Y$ are smooth sections of $T\OO^\perp$, then $\omega(X) = \omega(Y) = 0$ and so we have
\[
    \Omega(X,Y) = X(\omega(Y)) - Y(\omega(X)) - \omega([X,Y]) = -\omega([X,Y]),
\]
which implies the following result (see \cite{Gromov,QuirogaCh}).
\begin{lemma}
    \label{lem:TOperp-Omega}
    For $G$ and $M$ as above, assume that $T\widetilde{M} = T\OO \oplus T\OO^\perp$. Then, $T\OO^\perp$ is integrable if and only if $\Omega \equiv 0$.
\end{lemma}

Assume that $T\widetilde{M} = T\OO \oplus T\OO^\perp$, which implies that for every $x \in \widetilde{M}$ the subspace $T_x\OO^\perp$ is non-degenerate with respect to the scalar product of $T_x \widetilde{M}$. Hence Lemma~\ref{lem:so(E)} yields from the linear map $\Omega_x : \wedge^2 T_x\OO^\perp \rightarrow \DD(\g)$ a corresponding map $\so(T_x\OO^\perp) \rightarrow \DD(\g)$ given by
\[
    \Omega_x \circ \varphi_x^{-1}
\]
where $\varphi_x : \wedge^2 T_x\OO^\perp \rightarrow \so(T_x\OO^\perp)$ is the isomorphism defined by Lemma~\ref{lem:so(E)}. This does not change the $\so(T_x\OO^\perp)$-module structure on the domain. Hence, we will denote with the same symbol $\Omega_x$ the linear map given by the $2$-form $\Omega$ when considered as a map $\so(T_x\OO^\perp) \rightarrow \DD(\g)$.

It turns out that the forms $\omega_x$ and $\Omega_x$ have intertwining properties with respect to the module structure over $\GG(x) + \HH_0(x)$.

\begin{proposition}
    \label{prop:Omega-intertwining}
    For $G$ and $M$ as above, let $A$ and $U$ be as in Propositions~\ref{prop:Dg(x)} and \ref{prop:GromovOpenDense}, respectively. Assume that $T\widetilde{M} = T\OO \oplus T\OO^\perp$. For a fixed point $x \in A \cap U$, the following properties hold.
    \begin{enumerate}
        \item For every $X \in \DD(\g)$ and $Y \in \X(\widetilde{M})$ we have
            \[
                \omega_x([\rho_x(X), Y]_x) = [X, \omega_x(Y)].
            \]
        \item The linear map $\Omega_x : \wedge^2 T_x \OO^\perp \rightarrow \DD(\g)$ intertwines the homomorphism of Lie algebras $\widehat{\rho}_x : \g \rightarrow \GG(x)$ for the actions of $\g$ on $\DD(\g)$ and of $\GG(x)$ on $T_x\OO^\perp$ via $\lambda_x^\perp$. More precisely, we have
            \[
                [X, \Omega_x(u\wedge v)]
                = \Omega_x\left(\lambda_x^\perp(\widehat{\rho}_x(X))(u\wedge v)\right)
            \]
            for every $X \in \g$ and $u,v \in T_x\OO^\perp$.
        \item The linear map $\Omega_x : \so(T_x\OO^\perp) \rightarrow \DD(\g)$ is $\HH_0(x)$-invariant via $\lambda_x^\perp$. More precisely, we have
            \[
                \Omega_x\left([\lambda_x^\perp(Z), T]\right) = 0,
            \]
            for every $Z \in \HH_0(x)$ and $T \in \so(T_x\OO^\perp)$. In other words, we have
            \[
                [\lambda_x^\perp(\HH_0(x)), \so(T_x\OO^\perp)] \subset \ker(\Omega_x).
            \]
    \end{enumerate}
\end{proposition}
\begin{proof}
    In what follows, for any vector field $Y$ we will denote with $Y^\top$ and $Y^\perp$ the sections of $T\OO$ and $T\OO^\perp$, respectively, such that $Y = Y^\top + Y^\perp$.

    To prove (1), fix $X \in \DD(\g)$ and $Y \in \X(\widetilde{M})$. Since $\rho_x(X)$ preserves both $T\OO$ and $T\OO^\perp$, it follows that $[\rho_x(X), Y^\top]$ and $[\rho_x(X), Y^\perp]$ are sections of $T\OO$ and $T\OO^\perp$, respectively. We also note that $Y^\top_x = \omega_x(Y)^*_x$. On the other hand, $\rho_x(X)$ vanishes at $x$ and the dependence of $[\rho_x(X),Y^\top]_x$ with respect to $Y^\top$ is only on $Y^\top_x$. Hence, we have the following identities
    \begin{align*}
        \omega_x([\rho_x(X), Y]_x) &= \omega_x([\rho_x(X), Y^\top]_x) \\
            &= \omega_x([\rho_x(X), \omega_x(Y)^*]_x) \\
            &= \omega_x([X, \omega_x(Y)]^*_x) \\
            &= [X, \omega_x(Y)],
    \end{align*}
    where we have used in the third equality Proposition~\ref{prop:Dg(x)}(3).

    For (2), we consider the interpretation of $\Omega_x$ as a bilinear form and prove that
    \[
    [X, \Omega_x(u,v)] = \Omega_x\left(\lambda_x^\perp(\widehat{\rho}_x(X))(u),v\right) +
                \Omega_x\left(u,\lambda_x^\perp(\widehat{\rho}_x(X))(v)\right)
    \]
    for every $X \in \g$ and $u,v \in T_x\OO^\perp$. The identity is trivial for $X \in Z(\g)$, so we will assume that $X \in \DD(\g)$. Let $Y_1,Y_2 \in \HH$ such that $Y_1(x) = u$ and $Y_2(x) = v$. Then, by definition we have
    \[
        \lambda_x^\perp(\widehat{\rho}_x(X))(u) = [\widehat{\rho}_x(X), Y_1]_x
        = [\rho_x(X) + X^*, Y_1]_x = [\rho_x(X), Y_1]_x
    \]
    and similarly we have
    \[
        \lambda_x^\perp(\widehat{\rho}_x(X))(v) = [\rho_x(X), Y_2]_x.
    \]
    We now choose $\widehat{Y}_1, \widehat{Y}_2$ sections of $T\OO^\perp$ such that $\widehat{Y}_1(x) = u$ and $\widehat{Y}_2(x) = v$. As remarked above, since $\rho_x(X)$ vanishes at $x$ we have
    \begin{align*}
        &\lambda_x^\perp(\widehat{\rho}_x(X))(u) = [\rho_x(X), Y_1]_x = [\rho_x(X), \widehat{Y}_1]_x, \\
        &\lambda_x^\perp(\widehat{\rho}_x(X))(v) = [\rho_x(X), Y_2]_x = [\rho_x(X), \widehat{Y}_2]_x.
    \end{align*}
    Using the above we now compute
    \begin{align*}
        \Omega_x\left(\lambda_x^\perp(\widehat{\rho}_x(X))(u),v\right) &+
                \Omega_x\left(u,\lambda_x^\perp(\widehat{\rho}_x(X))(v)\right) \\
            &= \Omega_x\left([\rho_x(X), \widehat{Y}_1]_x,\widehat{Y}_2(x)\right) +
                \Omega_x\left(\widehat{Y}_1(x),[\rho_x(X), \widehat{Y}_2]_x\right) \\
            &= -\omega_x\left([[\rho_x(X), \widehat{Y}_1],\widehat{Y}_2]_x\right)
                -\omega_x\left([\widehat{Y}_1,[\rho_x(X), \widehat{Y}_2]]_x\right) \\
            &= -\omega_x\left([\rho_x(X), [\widehat{Y}_1,\widehat{Y}_2]]_x\right) \\
            &= -[X, \omega_x([\widehat{Y}_1, \widehat{Y}_2])] \\
            &= [X, \Omega_x(\widehat{Y}_1(x), \widehat{Y}_2(x))] \\
            &= [X, \Omega_x(u,v)],
    \end{align*}
    where we have used (1) in the fourth equality.

    To prove (3) we observe that it is enough to show that
    \[
        \Omega_x\left(\lambda_x^\perp(Z)(u),v\right) + \Omega_x\left(u,\lambda_x^\perp(Z)(v)\right) = 0,
    \]
    for any given $Z \in \HH_0(x)$, $u, v \in T_x\OO^\perp$, i.e.~we can consider $\Omega_x$ as a linear map $\wedge^2 T_x\OO^\perp \rightarrow \DD(\g)$. This is the case by the above remarks on Lemma~\ref{lem:so(E)}, which imply that $\varphi_x : \wedge^2 T_x\OO^\perp \rightarrow \so(T_x\OO^\perp)$ is an isomorphism of $\HH_0(x)$-modules via the representation $\lambda_x^\perp : \GG(x) + \HH_0(x) \rightarrow \so(T_x\OO^\perp)$.

    Given $Z \in \HH_0(x)$ and $u,v \in T_x\OO^\perp$, we start by choosing vector fields $Y_1, Y_2, \widehat{Y}_1, \widehat{Y}_2$ as above: $Y_1, Y_2$ belong to $\HH$, $\widehat{Y}_1, \widehat{Y}_2$ are sections of $T\OO^\perp$, $Y_1(x) = \widehat{Y}_1(x) = u$ and $Y_2(x) = \widehat{Y}_2(x) = v$. As in the proof of (2), since $Z_x = 0$, we have
    \begin{align*}
        &\lambda_x^\perp(Z)(u) = [Z, Y_1]_x = [Z, \widehat{Y}_1]_x, \\
        &\lambda_x^\perp(Z)(v) = [Z, Y_2]_x = [Z, \widehat{Y}_2]_x,
    \end{align*}
    and more generally, for any pair of vector fields $\widehat{W}$ and $W$ whose value at $x$ are the same we have
    \[
        [Z,\widehat{W}]_x = [Z, W]_x.
    \]
    Next, we observe that for any vector field $\widehat{W} \in \X(\widetilde{M})$ if we let $W \in \HH$ be such that $W_x = \widehat{W}^\top_x$, then
    \begin{align*}
        \overline{\omega}_x([Z,\widehat{W}]_x) &= \overline{\omega}_x([Z,\widehat{W}^\top]_x) \\
            &= \overline{\omega}_x([Z, W]_x) \\
            &= \overline{\omega}_x(\lambda_x(Z)(W_x)) \\
            &= \lambda_x(Z)(\overline{\omega}_x(W_x)) = 0,
    \end{align*}
    where we have used that $\overline{\omega}_x$ is a homomorphism of $\HH_0(x)$-modules and that $T_x\OO$ is a trivial $\HH_0(x)$-module. This relation in the case $\widehat{W} = [\widehat{Y}_1, \widehat{Y}_2]$ implies that
    \[
        0 = \overline{\omega}_x\left([Z,[\widehat{Y}_1,\widehat{Y}_2]]_x\right)
            = \overline{\omega}_x\left([[Z,\widehat{Y}_1],\widehat{Y}_2]_x\right)
                + \overline{\omega}_x\left([\widehat{Y}_1,[Z,\widehat{Y}_2]]_x\right)
    \]
    and applying $\alpha_x$ we obtain
    \begin{align*}
        0 &= \omega_x\left([[Z,\widehat{Y}_1],\widehat{Y}_2]_x\right)
            + \omega_x\left([\widehat{Y}_1,[Z,\widehat{Y}_2]]_x\right) \\
            &= -\Omega_x\left([Z,\widehat{Y}_1]_x,\widehat{Y}_2(x)\right)
                - \Omega_x\left(\widehat{Y}_1(x),[Z,\widehat{Y}_2]_x\right) \\
            &= -\Omega_x\left(\lambda_x^\perp(Z)(u),v\right) - \Omega_x\left(u,\lambda_x^\perp(Z)(v)\right),
    \end{align*}
    thus proving our last claim. Note that we have used in the second equality that, for $i=1,2$, the vector fields $\widehat{Y}_i$ and $[Z,\widehat{Y}_i]$ are sections of $T\OO^\perp$, and in the third identity the above formulas for $\lambda_x^\perp(Z)$ applied to $u,v$.
\end{proof}

For our subsequent analysis, we will consider the two cases given by the following result.

\begin{proposition}
    \label{prop:TOperp-integrable-q}
    Let $G$ and $M$ be as above, and assume that $T\widetilde{M} = T\OO \oplus T\OO^\perp$. Then, one of the following conditions is satisfied.
    \begin{enumerate}
    \item The normal bundle $T\OO^\perp$ is integrable.
    \item There is a dense conull subset $A_0 \subset \widetilde{M}$ contained in $A \cap U$, where $A$ and $U$ are given by Propositions~\ref{prop:Dg(x)} and \ref{prop:GromovOpenDense}, respectively, such that for every $x \in A_0$ the following properties are satisfied.
        \begin{enumerate}
            \item The linear map $\Omega_x : \wedge^2 T_x\OO^\perp \rightarrow \DD(\g)$ is surjective.
            \item The $\DD\GG(x)$-module structure on $T_x\OO^\perp$ is non-trivial.
            \item The homomorphism of Lie algebras $\lambda_x^\perp : \HH_0(x) \rightarrow \so(T_x\OO^\perp)$ is injective. Furthermore, $\lambda_x^\perp(\HH_0(x))$ is a $\lambda_x^\perp(\GG(x))$-submodule and a Lie subalgebra of $\so(T_x\OO^\perp)$ that satisfies
                \[
                    [\lambda_x^\perp(\HH_0(x)), \so(T_x\OO^\perp)] \subset \ker(\Omega_x).
                \]
        \end{enumerate}
    \end{enumerate}
    In particular, if (2) holds, then $T\OO^\perp$ is not integrable.
\end{proposition}
\begin{proof}
    Let us assume that $T\OO^\perp$ is not integrable. By Lemma~\ref{lem:TOperp-Omega} we have $\Omega \not= 0$, and since $\Omega$ is analytic, the set $A'$ of points $x \in \widetilde{M}$ where $\Omega_x \not= 0$ is the complement of a proper analytic subset. In particular, $A'$ is an open dense conull subset of $\widetilde{M}$. Then, we take $A_0 = A' \cap S \cap U$. Hence, (a) follows from Proposition~\ref{prop:Omega-intertwining}(2) and the fact that $\DD(\g)$ is simple. Now (b) follows from (a). Also, (c) is a restatement of Proposition~\ref{prop:H0(x)-lambda^perp} and Proposition~\ref{prop:Omega-intertwining}(3).

    Finally, if (2) holds, then from its part (a) it follows that $\Omega_x \not= 0$ for every $x \in A_0$ and so that $T\OO^\perp$ is not integrable.
\end{proof}

Case (1) of Proposition~\ref{prop:TOperp-integrable-q} has already been considered in \cite{QuirogaCh}. With this respect, the following is a consequence of Theorem~1.1 of \cite{QuirogaCh}.

\begin{proposition}
    \label{prop:TOperp-int-GtimesN}
    Let $G$ and $M$ be as above such that $T\widetilde{M} = T\OO \oplus T\OO^\perp$, and assume that $M$ is geodesically complete. If case (1) of Proposition~\ref{prop:TOperp-integrable-q} holds, then there exist
    \begin{enumerate}
        \item an isometric finite covering map $\widehat{M} \rightarrow M$ to which the $[G,G]$-action lifts,
        \item a simply connected pseudo-Riemannian manifold $N_1$,
        \item and a discrete subgroup $\Gamma \subset [G,G] \times \Iso(N_1)$,
    \end{enumerate}
    such that $\widehat{M}$ is $[G,G]$-equivariantly isometric to $([G,G] \times N_1)/ \Gamma$.
\end{proposition}

Hence, to complete the study of the structure of M it remains to consider case~(2)
of Proposition~\ref{prop:TOperp-integrable-q}.

\section{The centralizer of isometric $\wtU(p,q)$-actions}
\label{sec:Kill-Upq-actions}
In the rest of this work we will consider the case of $G = \wtU(p,q)$. Hence, we assume that there is an analytic isometric $\wtU(p,q)$-action on an analytic finite volume complete pseudo-Riemannian manifold $M$ so that there is a dense orbit for the group $[G,G] = \wtSU(p,q)$, and so that the $Z(\wtU(p,q))_0$-action is non-trivial and so locally free on a non-empty subset of $M$. In particular, the results from Section~\ref{sec:Kill-G-actions} apply for this setup. We also assume that $\dim M \leq \dim\U(p,q) + 2n$, with $p,q \geq 1$ and $n = p + q \geq 3$. Any non-zero element in $Z(\uni(p,q))$ yields an $\wtSU(p,q)$-invariant analytic Killing vector field on $M$. In what follows, $Z_0$ will denote a fixed Killing field so obtained and we will denote its lift to $\widetilde{M}$ by the same symbol. In particular, from Lemma~\ref{lem:G(x)} it follows that $\ZZ\GG(x) = \R Z_0$ for every $x \in A$, where $A$ is given as in Proposition~\ref{prop:Dg(x)}.

Given the above assumptions, we observe that by Lemma~\ref{lem:TOoplusTOperp} we have $TM = T\OO \oplus T\OO^\perp$. Furthermore, if case~(1) of Proposition~\ref{prop:TOperp-integrable-q} holds, then Proposition~\ref{prop:TOperp-int-GtimesN} describes the structure of M as a $\wtSU(p,q)$-space. Hence, we need to study the properties of $M$ when case~(2) of Proposition~\ref{prop:TOperp-integrable-q} holds.

An important ingredient is given by the following result, which will be repeatedly used below for several pairs of Lie algebras. Its proof is an easy consequence of the Jacobi identity.

\begin{lemma}
    \label{lem:brackets-homomorphism}
    Let $\h$ be a Lie algebra considered also as an $\h_1$-module where $\h_1$ is a semisimple Lie subalgebra of $\h$ and the module structure is given by the adjoint representation of $\h$ restricted to $\h_1$. Then, the linear map $\wedge^2\h \rightarrow \h$ defined by the Lie brackets is a homomorphism of $\h_1$-modules. As a consequence, for every pair $V_1, V_2$ of $\h_1$-submodules of $\h$, the $\h_1$-submodule $[V_1, V_2]$ $\mathrm{(}[V_1, V_1]$$\mathrm{)}$ is an $\h_1$-module quotient of the $\h_1$-module $V_1 \otimes V_2$ $\mathrm{(}$$\wedge^2 V_1$, respectively$\mathrm{)}$. In particular, $[V_1, V_2]$ $\mathrm{(}[V_1, V_1]$$\mathrm{)}$ lies in a sum of irreducible $\h_1$-submodules of $\h$ isomorphic to those appearing in the decomposition into irreducible submodules of the $\h_1$-module $V_1 \otimes V_2$
    $\mathrm{(}$$\wedge^2 V_1$, respectively$\mathrm{)}$. Furthermore, such sum can be taken to
    containg each irreducible $\h_1$-submodule at most as many times as it appears in the
    $\h_1$-module $V_1 \otimes V_2$ $\mathrm{(}$$\wedge^2 V_1$, respectively$\mathrm{)}$.
\end{lemma}

For our setup, we have that $\GG(x) \simeq \uni(p,q)$ and $\DD\GG(x) \simeq \su(p,q)$ as Lie algebras. Thus, as remarked in Section~\ref{sec:Kill-G-actions} we will describe the isomorphism types of modules over $\GG(x)$ or $\DD\GG(x)$ in terms of known modules over $\uni(p,q)$ or $\su(p,q)$, respectively.

\begin{lemma}
    \label{lem:TOperpchoices}
    For the $\wtU(p,q)$-action on $M$ as above, assume that case (2) from Proposition~\ref{prop:TOperp-integrable-q} holds and let $A_0 \subset \widetilde{M}$ be given as in such case. Let $\VV(x)$ be a $\GG(x)$-submodule as given by Proposition~\ref{prop:HHwithVV(x)}. Then, for every $x \in A_0$ the next properties are satisfied.
    \begin{enumerate}
        \item If $(p,q) \not= (2,2)$, then $\dim M = \dim \SU(p,q) + 2n + 1$ and $\VV(x) \simeq \C^{p,q}_\R$ as $\DD\GG(x)$-modules. In particular, $T_x\OO^\perp \simeq \C^{p,q}_\R \oplus \R$ as $\DD\GG(x)$-modules and $\so(T_x\OO^\perp)$ is isomorphic as a Lie algebra to either $\so(2p,2q+1)$ or $\so(2p+1,2q)$.
        \item If $(p,q) = (2,2)$, one of the following holds, where $\R^k$ denotes the $k$-dimensional trivial module.
            \begin{enumerate}
                \item For some $k \in \{1,2,3\}$, $\dim M = \dim \SU(2,2) + 6 + k$ and $T_x\OO^\perp \simeq \R^{4,2}\oplus\R^k$ as $\DD\GG(x)$-modules.
                \item $\dim M = \dim \SU(2,2) + 9$ and $T_x\OO^\perp \simeq \C^{2,2}_\R \oplus \R$ as $\DD\GG(x)$-modules.
            \end{enumerate}
        \end{enumerate}
\end{lemma}
\begin{proof}
    Let us consider the decomposition into $\DD\GG(x)$-submodules
    \[
        \HH = \GG(x) \oplus \HH_0(x) \oplus \VV(x)
    \]
    given by Proposition~\ref{prop:HHwithVV(x)}. Also by this result we have that $ev_x(\ZZ\GG(x)) = \R Z_0(x)$ is a $1$-dimensional $\DD\GG(x)$-submodule (and so trivial) of $T_x\OO^\perp$ because $x \in A_0 \subset S\cap U$. Since $\DD\GG(x) \simeq \su(p,q)$, Lemma~\ref{lem:lowest-dim} implies that $T_x\OO^\perp$ contains a $\DD\GG(x)$-submodule isomorphic to either $\C^{p,q}_\R$ or $\R^{4,2}$ when $(p,q) = (2,2)$. This completes the proof of (2). To obtain (1), it remains to show that $\VV(x) \simeq \C^{p,q}_\R$ as $\DD\GG(x)$-modules, but this follows from the fact that, by Proposition~\ref{prop:HHwithVV(x)}, $ev_x$ maps $\VV(x)$ onto a submodule of $T_x\OO^\perp$ complementary to $\R Z_0(x)$.
\end{proof}

To describe the centralizer $\HH$ it remains to consider the possibilities for $\HH_0(x)$.

\begin{lemma}
    \label{lem:H0(x)-cases}
    For the $\wtU(p,q)$-action on $M$ as above, assume that $(p,q) \not= (2,2)$, that case (2) from Proposition~\ref{prop:TOperp-integrable-q} holds and let $A_0 \subset \widetilde{M}$ be given as in such case. Then, for every $x \in A_0$ one of the following holds.
    \begin{enumerate}
        \item $\HH_0(x) = 0$.
        \item $\HH_0(x) \simeq \R$ and $[\HH_0(x), \VV(x)] = \VV(x)$.
    \end{enumerate}
\end{lemma}
\begin{proof}
    By Propositions~\ref{prop:H0(x)-lambda^perp} and \ref{prop:Omega-intertwining}, the Lie subalgebra $\HH_0(x)$ is completely determined by its image $\lambda_x^\perp$ so that $\lambda_x^\perp(\HH_0(x))$ is a Lie subalgebra and a $\lambda_x^\perp(\DD\GG(x))$-submodule of $\so(T_x\OO^\perp)$ so that
    \[
        [\lambda_x^\perp(\HH_0(x)), \so(T_x\OO^\perp)] \subset \ker(\Omega_x).
    \]

    On the other hand, by Lemma~\ref{lem:TOperpchoices} we know that $\VV(x) \simeq \C^{p,q}_\R$, as $\DD\GG(x)$-modules, and by Lemma~\ref{lem:Cpq-scalar} the scalar product on $ev_x(\VV(x))$ inherited from $T_x\OO^\perp$ is (a nonzero multiple of) the canonical scalar product in $\C^{p,q}_\R$ which has signature $(2p,2q)$. The orthogonal complement of $ev_x(\VV(x))$ in $T_x\OO^\perp$ is necessarily $\R Z_0(x)$ and so the latter is nonnull for the scalar product of $T_x\OO^\perp$. This proves again the part of Lemma~\ref{lem:TOperpchoices}(1) which says that $\so(T_x\OO^\perp)$ is isomorphic to either $\so(2p,2q+1)$ or $\so(2p+1,2q)$. Furthermore, since this conclusion is obtained from the $\DD\GG(x)$-module structure, it also shows that we can assume that the embedding $\lambda_x^\perp : \DD\GG(x) \hookrightarrow \so(T_x\OO^\perp)$ is equivalent to the restriction to $\su(p,q)$ of one of the embeddings in Equation~\eqref{eq:su-so2}.

    Hence, Lemma~\ref{lem:su-so-decomp-psi} provides the $\DD\GG(x)$-module and Lie algebra structure of $\so(T_x\OO^\perp)$. From this we see that $\ker(\Omega_x)$ is the sum of the submodules complementary to $\lambda_x^\perp(\DD\GG(x))$. In particular, we have
    \[
        \ker(\Omega_x) \simeq \R \oplus (\wedge^2\C^n)_\R \oplus \C^{p,q}_\R
    \]
    Considering the above restrictions on $\lambda_x^\perp(\HH_0(x))$ and the bracket identities from Lemmas~\ref{lem:su-so-decomp-phi} and \ref{lem:su-so-decomp-psi}, we conclude that the only possibilities for $\HH_0(x)$ are to be $0$ or $1$-dimensional.

    Let us suppose that $\HH_0(x)$ is $1$-dimensional. Since $\HH_0(x) \otimes \VV(x) \simeq \VV(x)$ as $\DD\GG(x)$-modules, by Lemma~\ref{lem:brackets-homomorphism} it follows that $[\HH_0(x), \VV(x)] \subset \VV(x)$. If the latter inclusion is proper, then the irreducibility of $\VV(x)$ implies that $[\HH_0(x), \VV(x)] = 0$. Hence, the fact that $T_x\OO^\perp = ev_x(\ZZ\GG(x) \oplus \VV(x))$ (see Proposition~\ref{prop:HHwithVV(x)}) together with the definition of $\lambda_x^\perp$ and its injectivity on $\HH_0(x)$ imply that $\HH_0(x) = 0$. This contradiction shows that $[\HH_0(x), \VV(x)] = \VV(x)$.
\end{proof}

We now describe the structure of the centralizer $\HH$ using Lemma~\ref{lem:H0(x)-cases}.

\begin{proposition}
    \label{prop:HH-structure-nonintTOperp}
    For the $\wtU(p,q)$-action on $M$ as above, assume that $(p,q) \not= (2,2)$, that case (2) from Proposition~\ref{prop:TOperp-integrable-q} holds and let $A_0 \subset \widetilde{M}$ be given as in such case. Then, for every $x \in A_0$ we have $\rad(\HH) = \ZZ\GG(x)$, $\HH_0(x) \simeq \R$ and there exists a $1$-dimensional $\DD\GG(x)$-submodule $\LL$ of $\ZZ\GG(x) \oplus \HH_0(x)$ such that
    \begin{enumerate}
        \item $\LL \not= \ZZ\GG(x)$.
        \item $(\mS, \DD\GG(x) \oplus \LL)$ is a symmetric pair equivalent to one of the symmetric pairs $(\su(p,q+1), \uni(p,q))$ or $(\su(p+1,q), \uni(p,q))$.
    \end{enumerate}
    In particular, $\HH$ is isomorphic to either $\uni(p,q+1)$ or $\uni(p+1,q)$.
\end{proposition}
\begin{proof}
    Choose and fix an element $x \in A_0$.

    By Lemma~\ref{lem:H0(x)-cases}, the only possibilities for $\HH_0(x)$ is to be $0$ or $1$-dimensional. Let us consider first the case $\HH_0(x) = 0$.

    Let $\mS$ be a Levi factor that contains $\DD\GG(x)$. Hence, both $\mS$ and $\rad(\HH)$ are $\DD\GG(x)$-submodules and so they are given as a sum of the subspaces $\DD\GG(x)$, $\ZZ\GG(x)$ and $\VV(x)$. On the other hand, by Lemma~\ref{lem:G(x)} we have $[\ZZ\GG(x), \HH] = 0$ and so $\ZZ\GG(x) \subset \rad(\HH)$. In particular, we have
    \begin{align*}
        \DD\GG(x) &\subset \mS \subset \DD\GG(x) \oplus \VV(x) \\
        \ZZ\GG(x) &\subset \rad(\HH) \subset \ZZ\GG(x) \oplus \VV(x).
    \end{align*}
    Thus, the cases to consider are whether or not $\mS$ contains $\VV(x)$.

    Suppose that $\VV(x)$ is contained in $\mS$. Every ideal of $\mS$ is a $\DD\GG(x)$-submodule and so a sum of the submodules $\DD\GG(x)$ and $\VV(x)$. Since $[\DD\GG(x), \VV(x)] = \VV(x)$ any simple ideal that contains $\DD\GG(x)$ also contains $\VV(x)$. This implies that $\mS$ is a simple Lie algebra. On the other hand, by Lemmas~\ref{lem:brackets-homomorphism} and \ref{lem:wedgeCpqR} the $\DD\GG(x)$-submodule $[\VV(x), \VV(x)]$ is either $0$ or a sum of modules isomorphic to either $\R$, $\su(p,q)$ or $(\wedge^2 \C^n)_\R$. This implies that $[\VV(x), \VV(x)] \subset \DD\GG(x)$. Hence, $(\mS, \DD\GG(x))$ is a symmetric pair with $\DD\GG(x) \simeq \su(p,q)$ as Lie algebra and $\mS/\DD\GG(x) \simeq \C^{p,q}_\R$ as $\DD\GG(x)$-module. An inspection of Table~II from \cite{Berger} shows that no such symmetric pair exists. This contradiction shows that $\VV(x)$ is not contained in $\mS$ and so
    \begin{align*}
        \mS &= \DD\GG(x) \\
        \rad(\HH) &= \ZZ\GG(x) \oplus \VV(x),
    \end{align*}
    from which we obtain the decomposition
    \begin{equation}\label{eq:semidirectproduct}
        \HH = \rad(\HH) \rtimes \DD\GG(x).
    \end{equation}

    Now assume that $\HH_0(x) \simeq \R$.

    Let $\mS$ be a Levi factor of $\HH$ such that $\DD\GG(x) \subset \mS$. By Lemma~\ref{lem:G(x)} we know that $\ZZ\GG(x)$ is an Abelian ideal of $\HH$ and so it is contained in $\rad(\HH)$. Hence we have
    \begin{align*}
        \DD\GG(x) &\subset \mS \\
        \ZZ\GG(x) &\subset \rad(\HH) \subset \ZZ\GG(x) \oplus \HH_0(x) \oplus \VV(x),
    \end{align*}
    so that both $\mS$ and $\rad(\HH)$ are sums of irreducible $\DD\GG(x)$-submodules of $\HH$. The possibilities to consider are the following.
    \begin{itemize}
        \item $\rad(\HH) = \ZZ\GG(x) \oplus \HH_0(x)$.
        \item $\rad(\HH) = \ZZ\GG(x) \oplus \VV(x)$.
        \item $\rad(\HH) = \ZZ\GG(x) \oplus \HH_0(x) \oplus \VV(x)$.
        \item $\rad(\HH) = \ZZ\GG(x)$.
    \end{itemize}
    Note that one has to consider the cases where $\rad(\HH)$ contains a $1$-dimensional $\DD\GG(x)$-submodule $\LL$ of $\ZZ\GG(x) \oplus \HH_0(x)$ different from $\ZZ\GG(x)$. But if $\rad(\HH)$ contains such $\LL$, then it contains $\HH_0(x)$ as well since it already contains $\ZZ\GG(x)$. And so the above are indeed all the cases to consider.

    If $\rad(\HH) = \ZZ\GG(x) \oplus \HH_0(x)$, then we have $\mS = \DD\GG(x) \oplus \VV(x)$. As above this yields a contradiction. If $\rad(\HH) = \ZZ\GG(x) \oplus \VV(x)$, then we have $\mS = \DD\GG(x) \oplus \LL$ for some $1$-dimensional $\DD\GG(x)$-submodule $\LL$ of $\ZZ\GG(x) \oplus \HH_0(x)$. This yields a $1$-dimensional ideal $\LL$ of $\mS$ which is also a contradiction.

    For the case $\rad(\HH) = \ZZ\GG(x) \oplus \HH_0(x) \oplus \VV(x)$ we have $\mS = \DD\GG(x)$ and so the decomposition given by Equation~\eqref{eq:semidirectproduct} holds again. We will proceed to prove that the decomposition from Equation~\eqref{eq:semidirectproduct} yields a contradiction.

    Choose $R$ a simply connected Lie group whose Lie algebra is $\rad(\HH)$. Hence, the product space of the Lie groups $R$ and $\wtSU(p,q)$ has a semidirect Lie group structure whose Lie algebra is isomorphic to $\HH = \rad{\HH} \rtimes \DD\GG(x) \simeq \rad{\HH} \rtimes \su(p,q)$. We will denote with $R \rtimes \wtSU(p,q)$ this semidirect product. For this construction we have considered the isomorphism $\su(p,q) \simeq \DD\GG(x)$ given by $\widehat{\rho}_x$ from Lemma~\ref{lem:G(x)}. Let us denote by $\psi : \rad{\HH} \rtimes \su(p,q) \rightarrow \HH$ the isomorphism thus considered. Since $\HH \subset \Kill(\widetilde{M})$, by Lemma~1.11 from \cite{OQ-SO} (see also \cite{ONeill}) there exists an isometric right $R\rtimes \wtSU(p,q)$-action on $\widetilde{M}$ such that
    \[
        \psi(X) = X^*
    \]
    for every $X \in \rad(\HH) \rtimes \su(p,q)$. Note that for $X \in \rad(\HH) \rtimes \su(p,q)$ we denoted by $X^*$ the Killing field on $\widetilde{M}$ whose local flow is $(\exp(tX))_t$ by the right action of $R\rtimes \wtSU(p,q)$. Let $H_0$ be the connected subgroup of $R \rtimes \widetilde{SU}(p,q)$ whose Lie algebra is $\HH_0(x)$. Note that $H_0$ is a subgroup of the simply connected solvable subgroup $R$ and so it is closed in $R$ as well as in $R \rtimes \widetilde{SU}(p,q)$ (see Exercise D.4(vii) in Chapter~II from \cite{Helgason}). Consider the analytic map
    \begin{align*}
        f : H_0\backslash(R \rtimes \wtSU(p,q)) &\rightarrow \widetilde{M} \\
            H_0(r,g) &\mapsto x(r,g),
    \end{align*}
    which is clearly $R \rtimes \wtSU(p,q)$-equivariant. If we compute for $X \in \rad{\HH} \rtimes \su(p,q)$
    \begin{align*}
        df_{H_0(e,e)}(X + \HH_0(x)) &= \frac{d}{dt}\Big|_{t=0} x \exp(tX) \\
            &= X^*_x = ev_x(X^*) = ev_x(\psi(X))
    \end{align*}
    then we observe that $df_{H_0(e,e)}$ yields an isomorphism $\rad{\HH} \rtimes \su(p,q) \rightarrow T_{x_0}\widetilde{M}$ so that
    \begin{align*}
        df_{H_0(e,e)}(\su(p,q)) &= T_x\OO \\
        df_{H_0(e,e)}(\rad(\HH)/\HH_0(x)) &= ev_x(\rad(\HH)) = T_x\OO^\perp,
    \end{align*}
    which follows from the choice of $\psi$ and Proposition~\ref{prop:HHwithVV(x)}. This implies that $f$ is a local diffeomorphism at $H_0(e,e)$ and so everywhere by its $R \rtimes \wtSU(p,q)$-equivariance.

    Consider $H_0\backslash (R \rtimes e) \simeq H_0\backslash R$ and its image
    \[
        N = f(H_0\backslash (R \rtimes e)).
    \]
    In particular, $N$ is a submanifold of $\widetilde{M}$ in a neighborhood of $x$. We observe that we have
    \[
        T_x N = df_{H_0(e,e)}(\HH_0(x)\backslash\rad(\HH)) = T_x\OO^\perp.
    \]
    Furthermore, if we denote by $R(r,g)$ the right translation by $(r,g)$ on the spaces $H_0\backslash(R \rtimes \wtSU(p,q))$ and $\widetilde{M}$, then the equivariance of $f$ implies that for $r$ in a neighborhood of $e \in R$ we have
    \begin{align*}
        T_{f(H_0(r,e)))} N &= df_{f(H_0(r,e))}(T_{H_0(r,e)} H_0\backslash R) \\
            &= df_{f(H_0(r,e))}(dR(r,e)_{H_0(e,e)}(T_{H_0(e,e)} H_0\backslash R)) \\
            &= dR(r,e)_{f(H_0(r,e))}\circ df_{f(H_0(e,e))} (T_{H_0(e,e)}H_0\backslash R) \\
            &= dR(r,e)_{F(H_0(r,e))}(df_x(\HH_0(x)\backslash\rad(\HH))) \\
            &= dR(r,e)_{f(H_0(r,e))}(T_x\OO^\perp) \\
            &= T_{f(H_0(r,e))} \OO^\perp.
    \end{align*}
    Note that we have used in the last identity that the right $R \rtimes \wtSU(p,q)$-action commutes with the $\wtSU(p,q)$-action and so leaves invariant both $T\OO$ and $T\OO^\perp$. This shows that in neighborhood of $x$ the space $N$ defines an integral submanifold of $T\OO^\perp$ passing through $x$. Using again that the right $R \rtimes \wtSU(p,q)$-action leaves invariant $T\OO^\perp$ we conclude that for every $(e,g) \in R \rtimes \wtSU(p,q)$ in a neighborhood of $(e,e)$ the manifold $N(e,g)$ is also an integral submanifold of $T\OO^\perp$. Hence, the set
    \[
        \bigcup_{(g,e) \in R \rtimes \wtSU(p,q)} N (g,e) = f\left(H_0\backslash (R \rtimes \wtSU(p,q))\right)
    \]
    contains an open neighborhood of $x$ covered by integral submanifolds of $T\OO^\perp$. Hence, $T\OO^\perp$ is integrable everywhere by analyticity. Since we are assuming case (2) from Proposition~\ref{prop:TOperp-integrable-q} we already had that $T\OO^\perp$ is not integrable and so this yields a contradiction.

    It remains to consider the case where both $\HH_0(x) \simeq \R$ and $\rad(\HH) = \ZZ\GG(x)$ hold, for which we have
    \[
        \mS = \DD\GG(x) \oplus \LL \oplus \VV(x)
    \]
    for some $1$-dimensional $\DD\GG(x)$-submodule of $\ZZ\GG(x) \oplus \HH_0(x)$. Also note that $\LL \not= \ZZ\GG(x)$ since $\ZZ\GG(x)$ cannot be contained in $\mS$.

    If $\mS_1$ is an ideal of $\mS$ that contains $\DD\GG(x)$, then $\VV(x) \subset \mS_1$ because of the identity $[\DD\GG(x), \VV(x)] = \VV(x)$. Thus $\mS_1$ has codimension at most $1$ in $\mS$ and so $\mS_1 = \mS$. This implies that $\mS$ is a simple Lie algebra. On the other hand, by Lemmas~\ref{lem:brackets-homomorphism} and \ref{lem:wedgeCpqR} we have that
    \[
        [\VV(x), \VV(x)] \subset \DD\GG(x) \oplus \LL
    \]
    because $\VV(x) \simeq \C^{p,q}_\R$ as $\DD\GG(x)$-modules. Hence, $(\mS, \DD\GG(x) \oplus \LL)$ is a symmetric pair with $\mS$ a simple Lie algebra and $\mS/(\DD\GG(x) \oplus \LL) \simeq \C^{p,q}_\R$ as module over the Lie subalgebra $\DD\GG(x) \oplus \LL \simeq \uni(p,q)$. By Table~II from \cite{Berger} it follows that the symmetric pair $(\mS, \DD\GG(x) \oplus \LL)$ is equivalent to one of the symmetric pairs $(\su(p,q+1), \uni(p,q))$ or $(\su(p+1,q), \uni(p,q))$. This completes the proof of our statement.
\end{proof}

\section{Proof of the Main Results}
\label{sec:proofmainthms}
In the rest of this section we will assume that the hypotheses of Theorem~\ref{thm:difftype} are satisfied. By Proposition~\ref{prop:TOperp-integrable-q} we have two cases to consider according to whether $T\OO^\perp$ is integrable or not. If $T\OO^\perp$ is integrable, then Proposition~\ref{prop:TOperp-int-GtimesN} proves that the conclusions (1) from Theorems~\ref{thm:difftype} and \ref{thm:metrictype} hold. On the other hand, Proposition~\ref{prop:GH0-module-structure}(4) implies that $\ZZ\GG(x) \subset T_x\OO^\perp$ at every $x \in A\cap U$ and by Lemma~\ref{lem:G(x)}(2) we conclude that $Z^*_x \in T_x \OO^\perp$ for every $Z \in Z(\g)$ and for almost all (and so all) $x \in \widetilde{M}$. This shows that the $Z(\wtU(p,q))_0$-action preserves the factor $N_1$, which implies that the conclusion (1) from Theorem~\ref{thm:actiontype} holds.

Hence, we can assume that $T\OO^\perp$ is not integrable and so that case~(2) from Proposition~\ref{prop:TOperp-integrable-q} holds. Thus, we can choose and fix $x_0 \in A_0$ for which the description of the structure of the centralizer $\HH$ provided by Proposition~\ref{prop:HH-structure-nonintTOperp} holds. We will follow the notation of the latter. In particular, $\LL$ denotes the $1$-dimensional subspace of $\HH$ obtained in Proposition~\ref{prop:HH-structure-nonintTOperp}. Also note that
\[
    \mS = [\HH, \HH], \quad \ZZ\GG(x_0) = \rad(\HH).
\]

We consider in the following subsections the two cases according to whether $\LL$ equals $\HH_0(x_0)$ or not. The next result will be used for both.

\begin{lemma}
    \label{lem:ZY-bundles}
    Suppose that Case (2) from Proposition~\ref{prop:TOperp-integrable-q} holds and let $x_0 \in A_0$ be the above chosen point. Then, the action of $Z(\wtU(p,q))_0$ on $M$ is locally free with non-degenerate orbits for the metric of $M$. Furthermore, if we denote by $\YY$ the orthogonal complement of $T\OO \oplus \ZZ$, then $\YY$ is a subbundle of $TM$ that is non-degenerate with respect to the metric of $M$ and it also  satisfies $T\OO^\perp = \ZZ \oplus \YY$ and $\YY_{x_0} = ev_{x_0}(\VV(x_0))$.
\end{lemma}
\begin{proof}
    By Lemma~\ref{lem:TOperpchoices}(1) we have $T_{x_0} \OO^\perp \simeq \C^{p,q}_\R \oplus \R$ as $\DD\GG(x_0)$-modules. And so Proposition~\ref{prop:GH0-module-structure}(4) implies that the $\DD\GG(x_0)$-submodule of $T_{x_0}\OO^\perp$ isomorphic to $\R$ is necessarily $\ZZ_{x_0}$. Also, Lemma~\ref{lem:Cpq-scalar} together with the non-degeneracy of $T_{x_0}\OO^\perp$ imply that the scalar products on $\C^{p,q}_\R$ and $\R$ inherited from the isomorphism $T_{x_0} \OO^\perp \simeq \C^{p,q}_\R \oplus \R$ are both non-degenerate. In particular, the subspaces $\ZZ_{x_0}$ and $ev_{x_0}(\VV(x_0))$ are perpendicular and non-degenerate for the metric on $T_{x_0}\OO^\perp \subset T_{x_0}\widetilde{M}$. It follows from Proposition~\ref{prop:Zx} that the set $\ZZ$ of tangent spaces to the $Z(\wtU(p,q))_0$-orbits is an analytic line bundle with non-degenerate fibers everywhere. This proves the first claim, and the rest are now easy to conclude.
\end{proof}

In the rest of this section $\YY$ will denote the bundle defined in Lemma~\ref{lem:ZY-bundles}.

\subsection{Proof of Theorems~\ref{thm:difftype}, \ref{thm:metrictype} and \ref{thm:actiontype}: case $\LL \not= \HH_0(x_0)$.}
In this case the evaluation map $ev_{x_0} : \mS \rightarrow T_{x_0} \widetilde{M}$ is an isomorphism that maps realizing the following isomorphisms of $\DD\GG(x_0)$-modules
\[
    \DD\GG(x_0) \simeq T_{x_0} \OO, \quad
    \LL \simeq \ZZ_{x_0}, \quad
    \VV(x_0) \simeq \YY_{x_0}.
\]
Let us now denote by $\widetilde{S}$ a simply connected Lie group whose Lie algebra is $\mS$. By Lemma~1.11 from \cite{OQ-SO} (see also \cite{ONeill}) the geodesic completeness of $\widetilde{M}$ implies the existence of an isometric right $\widetilde{S}$-action on $\widetilde{M}$ such that for every $X \in \mS = Lie(\widetilde{S})$ the Killing vector field obtained by differentiating at $t=0$ the map
\[
    p \in \widetilde{M} \mapsto p\exp(tX)
\]
yields $X$ itself. Consider the $\widetilde{S}$-orbit map at $x_0$ given by
\begin{align*}
    \varphi : \widetilde{S} &\rightarrow \widetilde{M} \\
        \varphi(s) &= x_0 s.
\end{align*}
By the above remarks on $\mS$ we have
\[
    d\varphi_e(X) = ev_{x_0}(X)
\]
for every $X \in \mS$. Hence, the above mentioned properties of $ev_{x_0}$ imply that $d\varphi_e$ is an isomorphism. Since $\varphi$ is $\widetilde{S}$-equivariant for the right $\widetilde{S}$-action on itself, it follows that $\varphi$ is a local diffeomorphism.

The following result allows us to rescale the pseudo-Riemannian metric on $M$ to our needs.

\begin{lemma}
    \label{lem:Lnot=H0-rescaling}
    Assume that $\LL \not= \HH_0(x_0)$ and that the above notation holds. Let $h$ be the pseudo-Riemannian metric on $M$. Then, there exists a pseudo-Riemannian metric $\overline{h}$ on $M$ of the form
    \[
        \overline{h} = c_1 h|_{T\OO} \oplus c_2 h|_{\ZZ} \oplus c_3 h|_{\YY}.
    \]
    for some non-zero constants $c_1, c_2, c_3$ such that for $\mS$ endowed with the scalar product defined by its Killing form the map
    \[
        d\varphi_e = ev_{x_0} : \mS \rightarrow (T_{x_0}\widetilde{M}, \overline{h}_{x_0})
    \]
    is an isometry. Furthermore, the metric $\overline{h}$ is $\wtU(p,q)$-invariant on $M$ and its lift to $\widetilde{M}$ is $\widetilde{S}$-invariant.
\end{lemma}
\begin{proof}
    From the above, $d\varphi_e = ev_{x_0}$ is a homomorphism of $\DD\GG(x_0)$-modules that maps isomorphically
    \[
        \DD\GG(x_0) \simeq T_{x_0}\OO, \quad
        \LL \simeq \ZZ_{x_0}, \quad
        \VV(x_0) \simeq \YY_{x_0}.
    \]
    Furthermore, by Lemma~\ref{lem:TxM-module-structure} the $\DD\GG(x_0)$-action preserves the metric on $T_{x_0}\widetilde{M}$. Consider the scalar products on $\DD\GG(x_0)$, $\LL$ and $\VV(x_0)$ given by
    \[
        d\varphi_e^*(h|_{T_{x_0}\OO}), \quad
        d\varphi_e^*(h|_{\ZZ_{x_0}}), \quad
        d\varphi_e^*(h|_{\YY_{x_0}}),
    \]
    respectively. These scalar products are $\DD\GG(x_0)$-invariant as a consequence of the $\DD\GG(x_0)$-equivariance of $d\varphi_e = ev_{x_0}$. On the other hand, it is well known that, up to a multiplicative constant, the only $\DD\GG(x_0)$-invariant scalar products on the three spaces $\DD\GG(x_0)$, $\LL$ and $\VV(x_0)$ are the respective restrictions of the Killing form of $\mS$ to them. For $\VV(x_0)$ this is Lemma~\ref{lem:Cpq-scalar}, for $\DD\GG(x_0)$ this follows from the fact that $\DD\GG(x_0) \simeq \su(p,q)$ as Lie algebras and for $\LL$ it is a trivial claim. Hence, there exist non-zero constants $c_1, c_2, c_3$ such that the Killing form of $\mS$ is given by
    \[
        c_1 d\varphi_e^*(h|_{T_{x_0}\OO}) \oplus c_2 d\varphi_e^*(h|_{\ZZ_{x_0}}) \oplus c_3 d\varphi_e^*(h|_{\ZZ_{x_0}}),
    \]
    with respect to the decomposition $\mS = \DD\GG(x_0) \oplus \LL \oplus \VV(x_0)$. Here we have used that these three spaces are orthogonal and non-degenerate with respect to the Killing form of $\mS$. If we define the metric $\overline{h}$ as in our statement with these constants, then the first claim follows. More precisely, the map
    \[
        d\varphi_e = ev_{x_0} : \mS \rightarrow (T_{x_0}\widetilde{M}, \overline{h}_{x_0})
    \]
    is an isometry.

    On the other hand, since both $h$ and the decomposition $TM = T\OO \oplus \ZZ \oplus \YY$ are $\wtU(p,q)$-invariant, it follows that $\overline{h}$ is $\wtU(p,q)$-invariant as well. And by the same reason, the metric $\overline{h}$ lifted to $\widetilde{M}$ is also $\widetilde{S}$-invariant.
\end{proof}

Let us denote by $h_{\mS}$ the bi-invariant pseudo-Riemannian metric on $\widetilde{S}$ whose value at $e$ is the Killing form of its Lie algebra $\mS$. Hence, Lemma~\ref{lem:Lnot=H0-rescaling} proves that the $\widetilde{S}$-orbit map defined before
\[
    \varphi : (\widetilde{S}, h_{\mS}) \rightarrow (\widetilde{M}, \overline{h})
\]
has an isometric differential $d\varphi_e$ at $e$. Since $\varphi$ is $\widetilde{S}$-equivariant for the right action and since (by the same lemma) the $\widetilde{S}$-actions on both $(\widetilde{S}, h_{\mS})$ and $(\widetilde{M}, \overline{h})$ are isometric, we conclude that $\varphi$ is a local isometry. The geodesic completeness of $(\widetilde{S}, h_{\mS})$ implies by the results from \cite{ONeill} that $\varphi$ is an isometry. This proves that the conclusions (2) from Theorems~\ref{thm:difftype} and \ref{thm:metrictype} hold.

On the other hand, the $\wtU(p,q)$-action on $\widetilde{M}$ (lifted from that of $M$) yields through the isometry $\varphi$ an isometric left $\wtU(p,q)$-action on $\widetilde{S}$. Consider the corresponding homomorphism
\[
    \rho : \wtU(p,q) \rightarrow \Iso(\widetilde{S}).
\]
It is known (see for example \cite{Muller}) that the connected component of the identity of $\Iso(\widetilde{S})$ is $L(\widetilde{S})R(\widetilde{S})$, the group of left and right translations in the group $\widetilde{S}$. Hence, there exist homomorphisms $\rho_1, \rho_2 : \wtU(p,q) \rightarrow \widetilde{S}$ so that
\[
    \rho(g) = L_{\rho_1(g)} R_{\rho_2(g)^{-1}}
\]
for every $g \in \wtU(p,q)$. Recall that the left $\wtU(p,q)$-action and the right $\widetilde{S}$-action on $\widetilde{M}$ commute which implies that
\[
    \rho_2(\wtU(p,q)) \subset Z(\widetilde{S})
\]
and so that $\rho_2$ is trivial. This proves that $\rho$ can be thought as a non-trivial homomorphism $\wtU(p,q) \rightarrow \widetilde{S}$ and that the isometric $\wtU(p,q)$-action on $\widetilde{S}$ is the left translation action induced by this homomorphism

On the other hand, we know that the actions of $\wtSU(p,q)$ and $Z(\wtU(p,q))_0$ on $M$ are locally free, and so $\rho$ is locally injective on $\wtSU(p,q)$ and $Z(\wtU(p,q))_0$. Let $X \in \su(p,q)$ and $Y \in Z(\uni(p,q))$ be given so that
\[
    \rho(\exp(tX) \exp(tY)) = e
\]
for every $t \in \R$. By applying this to any given point in $\widetilde{S}$ and differentiating at $t = 0$ we conclude that $X^* = -Y^*$. Through the identification $\widetilde{S} \simeq \widetilde{M}$ given by $\varphi$ this yields the same identity $X^* = -Y^*$ for the corresponding vector fields in $\widetilde{M}$. Hence, the pointwise linear independence of $T\OO$ and $\ZZ$ implies that $X^* = Y^* = 0$, and the local freeness of the actions of $\wtSU(p,q)$ and $Z(\wtU(p,q))_0$ yield $X = Y = 0$. This proves that $\rho$ is a locally injective homomorphism on $\wtU(p,q)$. Consider the corresponding injective homomorphism of Lie algebras
\[
    d\rho : \uni(p,q) \rightarrow \mS.
\]
With respect to the $\su(p,q)$-module structure on $\mS$ given by the Lie subalgebra $d\rho(\su(p,q))$ let $V$ be an $\su(p,q)$-submodule such that
\[
    \mS = d\rho(\uni(p,q)) \oplus V.
\]
If $V$ is a trivial $\su(p,q)$-submodule, then Lemma~\ref{lem:brackets-homomorphism} implies that $d\rho(Z(\uni(p,q)) \oplus V$ is an ideal of $\mS$. This contradiction shows that $V$ is non-trivial and so, by Lemma~\ref{lem:lowest-dim}, it is isomorphic to $\C^{p,q}_\R$ as $\su(p,q)$-module. In this situation, Lemma~\ref{lem:brackets-homomorphism} now implies that
\[
    [V, V] \subset d\rho(\uni(p,q)),
\]
and so that $(\mS, d\rho(\uni(p,q)))$ is a symmetric pair. With the isomorphism types of the objects involved Table~II from \cite{Berger} shows that $(\mS, d\rho(\uni(p,q)))$ is equivalent as a symmetric pair to either $(\su(p,q+1), \uni(p,q))$ or $(\su(p+1,q), \uni(p,q))$. This proves that the homomorphism $\rho : \wtU(p,q) \rightarrow \widetilde{S}$ is a canonical local embedding of $\wtU(p,q)$ into either $\wtSU(p+1,q)$ or $\wtSU(p,q+1)$, according to whether $\widetilde{S}$ is isomorphic to one or the other. Since this defines the $\wtU(p,q)$-action on $\widetilde{S}$ this proves that the conclusion (2) from Theorem~\ref{thm:metrictype} holds.

\subsection{Proof of Theorems~\ref{thm:difftype}, \ref{thm:metrictype} and \ref{thm:actiontype}: case $\LL = \HH_0(x_0)$.} For this case, the map $ev_{x_0} : \HH \rightarrow T_{x_0}\widetilde{M}$ is a surjection with kernel $\LL = \HH_0(x_0)$ and maps isomorphically
\[
    \DD\GG(x_0) \simeq T_{x_0}\OO, \quad
    \ZZ\GG(x_0) \simeq \ZZ_{x_0}, \quad
    \VV(x_0) \simeq \YY_{x_0}.
\]
In this case we have
\[
    \mS = \DD\GG(x_0) \oplus \HH_0(x_0) \oplus \VV(x_0)
\]
As before, let us denote by $\widetilde{S}$ a simply connected Lie group whose Lie algebra is $\mS$. Hence, $\widetilde{S} \times \R$ is a simply connected Lie group whose Lie algebra is $\HH$, where the factor $\R$ corresponds to the Lie subalgebra $\ZZ\GG(x_0)$.

As in the previous subsection, there is an isometric analytic right $\widetilde{S}\times\R$-action on $\widetilde{M}$ so that for every $X \in \HH$ the Killing vector field over $\widetilde{M}$ determined by the local flow $p \mapsto p \exp(tX)$ is $X$ itself. Let us denote by $K$ the connected Lie subgroup of $\widetilde{S}$ with Lie algebra $\HH_0(x_0)$, which is a closed subgroup as one can easily see. Then, the map
\begin{align*}
    \varphi : K \backslash \widetilde{S} \times \R &\rightarrow \widetilde{M} \\
        \varphi(K(s,t)) &= x_0 (s,t)
\end{align*}
is well-defined and analytic. By the above remarks, we have
\[
    d\varphi_e(X + \HH_0(x_0)) = ev_{x_0}(X)
\]
for every $X \in \HH$, and so $d\varphi_e$ is an isomorphism because of the above remarks. The $\widetilde{S}\times\R$-equivariance of $\varphi$ thus implies that it is a local diffeomorphism.

We now rescale the metric as in the previous case.

\begin{lemma}
    \label{lem:L=H0-rescaling}
    Assume that $\LL = \HH_0(x_0)$ and that the above notation holds. Let $h$ be the pseudo-Riemannian metric on $M$. Consider $\HH = \mS \oplus Z(\HH)$ endowed with the scalar product obtained as the direct sum of the Killing form on $\mS$ and some scalar product on $Z(\HH)$. Then, there exists a pseudo-Riemannian metric $\overline{h}$ on $M$ of the form
    \[
        \overline{h} = c_1 h|_{T\OO} \oplus c_2 h|_{\ZZ} \oplus c_3 h|_{\YY}.
    \]
    for some non-zero constants $c_1, c_2, c_3$ such that the map
    \[
        d\varphi_e : \HH_0(x_0)\backslash\HH \rightarrow (T_{x_0}\widetilde{M}, \overline{h}_{x_0})
    \]
    is an isometry, where $\HH_0(x_0)\backslash\HH$ carries the quotient scalar product. Furthermore, the metric $\overline{h}$ is $\wtU(p,q)$-invariant on $M$ and its lift to $\widetilde{M}$ is $\widetilde{S}\times\R$-invariant.
\end{lemma}
\begin{proof}
    The proof is similar to that of Lemma~\ref{lem:Lnot=H0-rescaling}. The only additional argument needed is to observe that since $(\mS, \DD\GG(x_0) \oplus \LL)$ is a symmetric pair isomorphic to either $(\su(p,q+1), \uni(p,q))$ or $(\su(p+1,q), \uni(p,q))$ (see Proposition~\ref{prop:HH-structure-nonintTOperp}) it follows that $\LL = \HH_0(x_0)$ is non-degenerate with respect to the Killing form of $\mS$. Hence, the induced bilinear form on $\HH_0(x_0)\backslash\mS$ is indeed a scalar product.
\end{proof}

Let us denote by $h_{\HH}$ the bi-invariant pseudo-Riemannian metric on $\widetilde{S} \times \R$ whose value at $e$ is the scalar product on $\HH$ described in Lemma~\ref{lem:L=H0-rescaling}. Hence, $(\widetilde{S} \times \R, h_{\HH})$ is a geodesically complete pseudo-Riemannian manifold. Since the Lie algebra of $K$ is non-degenerate with respect to the metric on $\widetilde{S} \times \R$, it follows that $K$ is a non-degenerate submanifold with respect to the metric $h_\HH$. Hence, there is an induced pseudo-Riemannian metric on $K\backslash \widetilde{S} \times \R$ so that the quotient map
\[
    \pi : H \rightarrow K\backslash H
\]
is a pseudo-Riemannian submersion. We recall that the geodesics in $\widetilde{S}$ are translations of $1$-parameter subgroups, and it is easy to see that these are horizontal with respect to the pseudo-Riemannian submersion $\pi$. It follows from the results in \cite{ONeill} that $(K\backslash \widetilde{S} \times \R, h_{\HH})$ is geodesically complete. As in the previous case, we use Lemma~\ref{lem:L=H0-rescaling} to conclude that the $H$-orbit map defined before
\[
    \varphi : (K \backslash H, h_{\HH}) \rightarrow (\widetilde{M}, \overline{h})
\]
is an isometry. This proves that the conclusions (3) from Theorems~\ref{thm:difftype} and \ref{thm:metrictype} hold.

On the other hand, both the $\wtSU(p,q)$ and $Z(\wtU(p,q))_0$-actions preserve the decomposition $T\widetilde{M} = T\OO \oplus \YY \oplus \ZZ$. We also note that with respect to the diffeomorphism $\varphi$ the bundle $T\OO \oplus \YY$ restricted to $K \backslash \widetilde{S} \times \{0\}$ is precisely the tangent bundle of this manifold. And a similar relationship holds between $\ZZ$ and $\{K e\} \times \R$. As a consequence, with respect to the diffeomorphism $\varphi$, the $\wtU(p,q)$-action on $K \backslash \widetilde{S} \times \R$ is the product of isometric actions of $\wtSU(p,q)$ on $K \backslash \widetilde{S}$ and of $Z(\wtU(p,q))_0$ on $\R$. This shows that the conclusion (3) from Theorem~\ref{thm:actiontype} holds.

\subsection{Proof of Theorems~\ref{thm:irredcase} and \ref{thm:pi1(M)}}
From the above, Theorem~\ref{thm:irredcase} now follows from the definition of weak irreducibility which implies that $\widetilde{M}$ cannot be a non-trivial pseudo-Riemannian product when $M$ is weakly irreducible.

The first claim of Theorem~\ref{thm:pi1(M)} follows from Theorem~\ref{thm:metrictype}. For the second part, we assume that case (2) in Theorem~\ref{thm:metrictype} holds and we let
\[
    \Gamma = \pi_1(M) \cap \Iso_0(\widetilde{S}),
\]
where $\widetilde{S}$ is either $\wtSU(p,q+1)$ or $\wtSU(p+1,q)$ with a bi-invariant metric. Note that we have identified $N = \widetilde{S} \simeq \widetilde{M}$ through the $\wtSU(p,q)$-equivariant isometry $\varphi$. From the results in \cite{Muller} it follows that $\Iso_0(\widetilde{S})$ has finite index in $\Iso(\widetilde{S})$, and so $\Gamma$ has finite index in $\pi_1(M)$. Furthermore, we have $\Iso_0(\widetilde{S}) = L(\widetilde{S})R(\widetilde{S})$, the group generated by left and right translations. Hence, the inclusion $\Gamma \subset \Iso_0(\widetilde{S})$ is realized by a pair of homomorphisms
\[
    \rho_1, \rho_2 : \Gamma \rightarrow \widetilde{S},
\]
by the expression
\[
    \gamma = L_{\rho_1(\gamma)} \circ R_{\rho_2(\gamma)^{-1}}
\]
for every $\gamma \in \Gamma$. Next we observe that the $\Gamma$-action and the $\wtU(p,q)$-action on $\widetilde{M}$ lifted from $M$ commute with each other. Since we are assuming that (2) from Theorem~\ref{thm:actiontype} holds, the latter action is the left translation action for the canonical symmetric pair embedding of $\wtU(p,q)$ into $\widetilde{S}$. This implies that $\rho_1(\Gamma)$ is contained in the centralizer $C$ of $\wtU(p,q)$ in $\widetilde{S}$. Hence, through the above expression for the elements $\gamma \in \Gamma$, we conclude that $\Gamma$ is a subgroup of $C \times \widetilde{S}$.

\appendix
\section{Some facts on the Lie algebra $\su(p,q)$}
\label{app:supq}
The first task of this appendix is to determine the non-trivial irreducible real representations of
$\su(p,q)$ with lowest possible dimension. The main observation is that every irreducible real
representation of $\su(p,q)$ can be obtained from an irreducible complex representation of
$\sli(n,\C)$ according to the results from \cite{Oni}. Using this fact we obtain the following
result.

\begin{lemma}\label{lem:lowest-dim}
    Let $p, q \geq 1$ and $n = p+q \geq 3$.
    \begin{enumerate}
        \item \label{item:modules-not22} If $(p,q) \not= (2,2)$, then the lowest dimensional non-trivial irreducible real $\su(p,q)$-module is $\C^{p,q}_\R$. Furthermore, $\dim V > 2n + 1$ for every non-trivial irreducible real $\su(p,q)$-module $V$ nonisomorphic to $\C^{p,q}_\R$.
        \item \label{item:modules-22} The isomorphism $\su(2,2) \simeq \so(2,4)$ yields a $6$-dimensional irreducible real $\su(2,2)$-module. The latter and the $8$-dimensional irreducible real $\su(2,2)$-module $\C^{2,2}_\R$ are the lowest dimensional ones. In other words, if $V$ is any other non-trivial irreducible real $\su(2,2)$-module, then $\dim V > 9$.
    \end{enumerate}
\end{lemma}
\begin{proof}
    The following facts can be found in \cite{Oni}.
    \begin{itemize}
        \item Every irreducible complex $\sli(n,\C)$-module $W$ determines an irreducible real $\su(p,q)$-module $V$ through one of the following mutually exclusive possibilities.
            \begin{enumerate}
                \item[(a)] $V= W_\R$, the realification of $W$.
                \item[(b)] $V$ is a real form of $W$ which is $\su(p,q)$-invariant.
            \end{enumerate}
        \item Case (b) occurs precisely when $V$ is self-conjugate and its Cartan index, as defined in \cite{Oni}, is $1$.
        \item The above exhausts all possible irreducible real $\su(p,q)$-modules.
    \end{itemize}
    We will freely use the results from \cite{Oni} and refer to this work for further details.

    For a dominant weight $\lambda$ of $\sli(n,\C)$, we will denote by $V^\lambda$ and $W^\lambda$ the irreducible modules over $\su(p,q)$ and $\sli(n,\C)$, respectively, as described above. In particular, we have either $\dim V^\lambda = 2\dim_\C W^\lambda$ or $\dim V^\lambda = \dim_\C W^\lambda$ according to whether case (a) or (b), respectively, holds.

    Let us denote by $\omega_1, \dots, \omega_{n-1}$ the fundamental weights of $\sli(n,\C)$. We have the following two well known inequalities where $\lambda$ is a dominant weight different from a fundamental weight
    \begin{align*}
        \dim_\C W^\lambda &> \min_{i=1,\dots,n-1} \dim_\C W^{\omega_i} \\
        \min_{i=1,\dots,n-1} \dim_\C W^{\omega_i} &\geq \dim_\C
        W^{\omega_1} = \dim_\C W^{\omega_{n-1}} = n.
    \end{align*}
    Moreover, $\dim_\C W^{\omega_i} = n$ only for $i \in \{1, n-1\}$. By the results from \cite{Oni}, $W^{\omega_1}$ and $W^{\omega_{n-1}}$ define the same irreducible real $\su(p,q)$-module $V^{\omega_1} = \C^{p,q}_\R$.

    The above remarks imply that
    \[
        \dim V^\lambda \geq 2n + 2
    \]
    for every dominant weight $\lambda \not= \omega_1, \omega_{n-1}$ whose corresponding $\sli(n,\C)$-module is not self-conjugate. To prove \eqref{item:modules-not22}, it remains to consider the dominant weights whose corresponding modules are self-conjugate.

    By \cite{Oni}, a dominant weight
    \[
        \lambda = k_1 \omega_1 + \dots + k_{n-1} \omega_{n-1}
    \]
    defines a self-conjugate $\sli(n,\C)$-module precisely when it satisfies $k_i = k_{n-i}$ for every $i = 1, \dots, [n/2]$. Note that in this case, one still has to consider the Cartan index in order to determine if the irreducible real representation of $\su(p,q)$ comes from either case (a) or (b).

    Let us define the dominant weights given by $\lambda_i = \omega_i + \omega_{n-i}$ for $i=1, \dots, [n/2]-1$ and
    \[
        \lambda_{[n/2]} =
            \begin{cases}
                \omega_{n/2} & \text{ if $n$ is even;} \\
                \omega_{[n/2]} + \omega_{[n/2]+1} & \text{ if $n$ is odd.}
            \end{cases}
    \]
    These weights yield self-conjugate irreducible $\sli(n,\C)$-modules. Furthermore, since Weyl dimension formula for $W^\lambda$ is monotone in each one of the coordinates $k_i$ of $\lambda$ as above, we conclude that
    \[
        \dim_\C W^\lambda \geq \min_{i=1,\dots,[n/2]} \dim_\C W^{\lambda_i}.
    \]
    Hence, we will consider the value of the right-hand side of this last inequality.

    \begin{claim}[1]
    For every $n = p+q \not= 4$ we have the following.
    \begin{itemize}
        \item If $n$ is odd, then
        \[
            \min_{i=1,\dots,[n/2]} \dim_\C W^{\lambda_i} \geq \dim_\C W^{\lambda_1}.
        \]
        \item If $n=2m$ is even, then
            \begin{align*}
                \min_{i=1,\dots,m-1} \dim_\C W^{\lambda_i} &\geq \dim_\C W^{\lambda_1}, \\
                \dim_\C W^{\lambda_m} = \dim_\C W^{\omega_m} &> 2n + 1.
            \end{align*}
    \end{itemize}
    Furthermore, for every $(p,q)$ as in our hypotheses we have $\dim_\C W^{\lambda_1} > 2n + 1$.
    \end{claim}

    To prove Claim (1), we will use the real form of the Cartan subalgebra of $\sli(n,\C)$ that consists of diagonal matrices. Such real form is naturally identified with the real vector space $\h = \{ v \in \R^n : v_1 + \dots + v_n = 0 \}$ and endowed with the canonical scalar product $(\cdot,\cdot)$ inherited from $\R^n$. For $v \in \R^n$, we will denote by $v^\perp$ its orthogonal projection into $\h$. In particular, the fundamental weights are given by $\omega_i = (e_1 + \dots + e_i)^\perp$, for $i = 1, \dots, n-1$, where $(e_i)_{i=1}^n$ is the canonical base of $\R^n$. Also, the sum of positive roots is given by
    \[
        \rho = \frac{1}{2}(n-1, n-3, \dots, -n+1).
    \]
    First we compute $\dim_\C W^{\lambda_1}$. In this case we have highest weight
    \[
        \lambda_1 = \omega_1 + \omega_{n-1} = (2e_1 + e_2 + \dots + e_{n-1})^\perp,
    \]
    and a direct application of Weyl dimension formula yields
    \begin{align}
        \dim_\C W^{\lambda_1}
        &= \prod_{1\leq \nu < \mu \leq n} \left(
            1 + \frac{(\lambda_1, e_\nu-e_\mu)}{(\rho, e_\nu-e_\mu)}
            \right) \label{eq:dim1}\\
        &= \prod_{\mu=2}^{n-1} \left( 1 + \frac{1}{\mu-1} \right)
            \left( 1 + \frac{2}{n-1} \right)
            \prod_{\nu=2}^{n-1} \left( 1 + \frac{1}{n-\nu} \right) \nonumber \\
        &= \prod_{\mu=2}^{n-1} \left( \frac{\mu}{\mu-1} \right)
            \left( \frac{n+1}{n-1} \right) \prod_{\nu=2}^{n-1} \left(
                \frac{n-\nu+1}{n-\nu} \right) \nonumber \\
        &= (n-1) \left( \frac{n+1}{n-1} \right)
            (n-1) = n^2 - 1, \nonumber
    \end{align}
    from which we conclude that $\dim_\C W^{\lambda_1} > 2n + 1$ for $n \geq 3$, thus proving the last part of Claim (1).

    We now proceed to prove the two items in the first part of Claim~(1), so we assume that $n \not= 4$. Choose a positive integer $i < n/2$ and consider the dominant weight
    \[
        \lambda_i = \omega_i + \omega_{n-i} = (2(e_1 + \dots + e_i) + e_{i+1} + \dots + e_{n-i})^\perp,
    \]
    for which a direct application of Weyl dimension formula yields
    \begin{align}
        \dim_\C W^{\lambda_i}
            &= \prod_{1\leq \nu < \mu \leq n}
                \left( 1 + \frac{(\lambda_i, e_\nu-e_\mu)}{(\rho, e_\nu-e_\mu)} \right) \label{eq:dimi} \\
            &= \prod_{\nu=1}^i\prod_{\mu=i+1}^{n-i}
                \left( 1 + \frac{1}{\mu-\nu} \right) \nonumber \\
            &\times \prod_{\nu=1}^i\prod_{\mu=n-i+1}^n
                \left( 1 + \frac{2}{\mu-\nu} \right) \nonumber \\
            &\times \prod_{\nu=i+1}^{n-i}\prod_{\mu=n-i+1}^n
                \left( 1 + \frac{1}{\mu-\nu} \right). \nonumber
    \end{align}

    Let us denote by $Q_1, Q_2, Q_3$ the three factors in the second line of equation \eqref{eq:dim1} enumerated as they appear; in particular, $Q_1$ and $Q_3$ are indexed products and $Q_2$ has a single factor. Correspondingly, let $P_1, P_2, P_3$ be the three double products in the last three lines of equation \eqref{eq:dimi}, enumerated as they appear. To prove that $\dim_\C W^{\lambda_i} \geq \dim_\C W^{\lambda_1}$ it is enough to show that each of the factors $Q_k$ is bounded from above by subproducts taken from $P_1 P_2 P_3$ so that there are nonoverlaping subproducts. This is achieved with the following choices.
    \begin{itemize}
        \item From $P_1$ the subproduct given by fixing $\mu = i+1$ and varying $\nu = 2, \dots, i$ yields
            \[
                \prod_{\nu=2}^i  \left( 1 + \frac{1}{i+1-\nu} \right)
                = \prod_{\mu=2}^i \left( 1 + \frac{1}{\mu-1} \right).
            \]
            In $P_1$ the subproduct given by fixing $\nu = 1$ is precisely
            \[
                \prod_{\mu=i+1}^{n-i} \left( 1 + \frac{1}{\mu - 1} \right).
            \]
            And the subproduct of $P_2$ obtained by fixing $\nu = 1$ and varying $\mu = n-i+1, \dots, n-1$ satisfies
            \[
                \prod_{\mu=n-i+1}^{n-1} \left( 1 + \frac{2}{\mu - 1} \right) >
                    \prod_{\mu=n-i+1}^{n-1} \left( 1 + \frac{1}{\mu - 1} \right).
            \]
            Hence, the product of the three subproducts just described bounds $Q_1$.
        \item The subproduct of $P_2$ obtained by fixing $\mu = n$ and varying $\nu = 2, \dots, i$ satisfies
            \[
                \prod_{\nu=2}^i \left( 1 + \frac{2}{n - \nu} \right) >
                    \prod_{\nu=2}^i \left( 1 + \frac{1}{n - \nu} \right).
            \]
            In $P_3$ the subproduct given by fixing $\mu = n$ is precisely
            \[
                \prod_{\nu=i+1}^{n-i} \left( 1 + \frac{1}{n - \nu} \right).
            \]
            And from $P_3$ the subproduct given by fixing $\nu = n - i$ and varying $\mu = n-i+1, \dots, n-1$ yields
            \[
                \prod_{\mu=n-i+1}^{n-1} \left( 1 + \frac{1}{\mu - n + i} \right)
                  = \prod_{\nu=n-i+1}^{n-1} \left( 1 + \frac{1}{n - \nu} \right).
            \]
            Hence, the product of the three subproducts just described bounds    $Q_3$.
        \item The factor in $P_2$ corresponding to $\nu = 1$ and $\mu = n$ equals $Q_2$.
    \end{itemize}

    The above proves the case of $n$ odd in Claim~(1) as well as the first inequality for the case
    of $n = 2m$ even. It remains to consider the dominant weight $\lambda_{n/2} = \lambda_m =
    \omega_m$.

    It is well known that:
    \[
        \dim_\C W^{\omega_m} =
            \begin{pmatrix}
                2m \\
                m
            \end{pmatrix}
            > 2^m
    \]
    where the last inequality holds for $m \geq 2$. And we observe that $2^m \geq 4m + 1 = 2n + 1$
    for $m \geq 5$.

    On the other hand, for $m = 3$, $n = 6$ we can directly verify that:
    \[
        \dim_\C W^{\omega_m} = \dim_\C W^{\omega_3} =
            \begin{pmatrix}
                6 \\
                3
            \end{pmatrix} = 20 > 13 = 4m + 1 = 2n + 1.
    \]
    And for $m = 4$, $n = 8$, we compute:
    \[
        \dim_\C W^{\omega_m} = \dim_\C W^{\omega_4} =
            \begin{pmatrix}
                8 \\
                4
            \end{pmatrix} = 70 > 17 = 4m + 1 = 2n + 1.
    \]

    This completes the proof of Claim (1), which together with above remarks imply part
    \eqref{item:modules-not22} of the Lemma for $\{p,q\} \not= \{1,3\}$. Hence, to complete the
    proof of the Lemma it remains to consider the case $n = 4$ for both $(p,q) = (2,2)$ and $(p,q)
    = (1,3)$.

    For $\lambda$ a self-conjugate dominant weight for $\sli(4,\C)$ and from the above discussion
    we can write $\lambda = k_1\lambda_1 + k_2 \omega_2$. If $k_1 \not= 0$, then the previous
    remarks and the first inequality of Claim (1) imply
    \[
        \dim_\C W^{\lambda} \geq \dim_\C W^{\lambda_1} > 2n + 1 = 9.
    \]
    In particular, $\dim V^{\lambda} > 2n + 1 = 9$ in this case.  Hence, we can assume that
    $\lambda = k \omega_2$, in which case Weyl dimension formula yields
    \[
        \dim_\C W^{k\omega_2} = \frac{(k+1)(k+2)^2(k+3)}{12}.
    \]
    It follows easily, that $\dim_\C W^{k\omega_2} \geq 20 > 9 = 2n + 1$ for $k \geq 2$. This again
    implies that $\dim V^{k \omega_2} > 2n + 1 = 9$ for every $k \geq 2$.

    Finally, it remains to consider the dimension of $V^{\omega_2}$ as real $\su(p,q)$-module for
    $(p,q) = (2,2)$ and $(p,q) = (1,3)$. From the above formulas, we have $\dim_\C W^{\omega_2} =
    6$, and from \cite{Oni} the Cartan index of $\omega_2$ is $1$ for the case $(p,q) = (2,2)$ and
    $-1$ for the case $(p,q) = (1,3)$. In particular, we have
    \[
        \dim V^{\omega_2} =
        \begin{cases}
            6 & \text{ for $(p,q) = (2,2)$}; \\
            12 & \text{ for $(p,q) = (1,3)$}.
        \end{cases}
    \]
    The result for the case $(p,q) = (1,3)$ completes the proof of case \eqref{item:modules-not22}
    of the Lemma. On the other hand, the $6$-dimensional real $\su(2,2)$-module $V^{\omega_2}$ is
    the defining representation that corresponds to the isomorphism $\su(2,2) \simeq \so(2,4)$;
    this proves case \eqref{item:modules-22}
    of the Lemma.
\end{proof}

Let us denote by $\left<\cdot,\cdot\right>_{p,q}$ the canonical Hermitian scalar product of
signature $p,q$ on $\C^{p,q}$. More precisely, we have
\[
    \left< z, w \right>_{p,q} = z^t I_{p,q} \overline{w}
\]
for $z, w \in \C^{p,q}$. By definition, this Hermitian form is $\su(p,q)$-invariant, and so it
defines two real bilinear forms on $\C^{p,q}_\R$ which are $\su(p,q)$-invariant as well. These are
given as the real and imaginary parts of $\left<\cdot,\cdot\right>_{p,q}$ and will be denoted as
follows
\begin{align*}
    \left<\cdot,\cdot\right>_0 &= \Re(\left<\cdot,\cdot\right>_{p,q}), &
    \omega_0(\cdot,\cdot) &= \Im(\left<\cdot,\cdot\right>_{p,q}).
\end{align*}
Note that these forms are non-degenerate and not a multiple of each other.

\begin{lemma}
    \label{lem:Cpq-scalar}
    Let $p, q \geq 1$. The space of $\su(p,q)$-invariant real bilinear forms on $\C^{p,q}_\R$ has
    dimension $2$ with a basis given by $\left<\cdot,\cdot\right>_0$ and $\omega_0(\cdot,\cdot)$.
    In particular, up to a constant, $\left<\cdot,\cdot\right>_0$ is the unique
    $\su(p,q)$-invariant symmetric real bilinear form on $\C^{p,q}_\R$.
\end{lemma}
\begin{proof}
    Let us denote by $\mathcal{B}$ the space of $\su(p,q)$-invariant real bilinear forms on
    $\C^{p,q}_\R$. Then, the map given by
    \begin{align*}
        f : \Hom_{\su(p,q)}(\C^{p,q}_\R) &\rightarrow \mathcal{B} \\
        T &\mapsto \left<T(\cdot),\cdot\right>_0
    \end{align*}
    defines an isomorphism of real vector spaces.

    By Schur's lemma, $\Hom_{\su(p,q)}(\C^{p,q}_\R)$ is a division algebra over $\R$, which has
    dimension at least $2$ by the remarks above. Hence, $\Hom_{\su(p,q)}(\C^{p,q}_\R)$ is
    isomorphic to either $\C$ or $\HQ$. The latter possibility would imply that $\C^{p,q}_\R$
    carries an $\su(p,q)$-invariant quaternionic structure; but this cannot be the case since
    $\C^n$ is not self-conjugate as an $\sli(n,\C)$-module (see \cite{Oni} for further details).
    This implies the first part of the statement, and the second follows readily.
\end{proof}

If we choose the real coordinates in $\C^{p,q}_\R$ ordered as $x_1, \dots, x_n, y_1, \dots, y_n$,
where $z = (z_1, \dots, z_n) \in \C^{p,q}_\R$ and $z_j = x_j + i y_j$, then the matrix
representation of $\left<\cdot, \cdot\right>_0$ is given by
\[
    \begin{pmatrix}
        I_{p,q} & 0 \\
        0 & I_{p,q}
    \end{pmatrix}.
\]
In what follows, $\so(2p,2q)$ will denote the pseudo-orthogonal Lie algebra for the latter matrix.
As noted above, the bilinear form $\left<\cdot, \cdot\right>_0$ is $\su(p,q)$-invariant. Thus, our
choice of coordinates realize the following embedding of Lie algebras.
\begin{align}
    \label{eq:su-so1}
    \varphi : \uni(p,q) &\rightarrow \so(2p,2q) \\
    A + i B &\mapsto
        \begin{pmatrix}
            A & -B \\
            B & A
        \end{pmatrix}, \nonumber
\end{align}
where $A, B$ are real matrices. Note that such $A,B$ necessarily satisfy:
\begin{align*}
    A^tI_{p,q} + I_{p,q}A &= 0 \\
    B^tI_{p,q} - I_{p,q}B &= 0.
\end{align*}

We will also consider the pseudo-orthogonal Lie algebras $\so(2p,2q+1)$ and $\so(2p+1,2q)$ given by
the following matrices,
\[
    \begin{pmatrix}
        I_{p,q} & 0 & 0 \\
        0 & I_{p,q} & 0 \\
        0 & 0 & -1
    \end{pmatrix}, \quad
    \begin{pmatrix}
        1 & 0 & 0 \\
        0 & I_{p,q} & 0 \\
        0 & 0 & I_{p,q}
    \end{pmatrix},
\]
respectively. Correspondingly, we define the following embeddings of Lie algebras.
\begin{align}
    \label{eq:su-so2}
    \psi_1 : \uni(p,q) &\rightarrow \so(2p,2q+1) &
    \psi_2 : \uni(p,q) &\rightarrow \so(2p+1,2q) \\
    A &\mapsto
    \begin{pmatrix}
        \varphi(A) & 0 \\
        0 & 0
    \end{pmatrix} &
    A &\mapsto
    \begin{pmatrix}
        0 & 0 \\
        0 & \varphi(A)
    \end{pmatrix}. \nonumber
\end{align}

We now describe the decomposition into irreducible submodules of the $\su(p,q)$-modules of $\so(2p,
2q)$, $\so(2p,2q+1)$ and $\so(2p+1,2q)$. In this work the $\su(p,q)$-module structure of these
spaces will always be considered as given by the embeddings from \eqref{eq:su-so1} and
\eqref{eq:su-so2}.

\begin{lemma}
    \label{lem:su-so-decomp-phi}
    Let $p, q \geq 1$ and $n = p+q \geq 3$. For $(p,q) \not= (2,2)$, the decomposition of
    $\so(2p,2q)$ into irreducible $\su(p,q)$-submodules is given by an isomorphism
    \[
    \so(2p,2q) \simeq \su(p,q) \oplus \R \oplus (\wedge^2\C^n)_\R.
    \]
    For $p=q=2$ the same isomorphism holds, but the last summand is not irreducible. Furthermore,
    for the induced Lie algebra structure on the righthand side of the above isomorphism we have
    \begin{align*}
        [\su(p,q) \oplus \R, (\wedge^2\C^n)_\R] &= (\wedge^2\C^n)_\R \\
        [(\wedge^2\C^n)_\R, (\wedge^2\C^n)_\R] &= \su(p,q) \oplus \R \\
        [\R, (\wedge^2\C^n)_\R] &= (\wedge^2\C^n)_\R.
    \end{align*}
\end{lemma}
\begin{proof}
    It follows from the fact that $(\so(2p,2q), \uni(p,q))$ is a symmetric pair together with its structure (see Table~II from \cite{Berger}).
\end{proof}

\begin{lemma}
    \label{lem:su-so-decomp-psi}
    Let $p, q \geq 1$ and $n = p+q \geq 3$. For $(p,q) \not= (2,2)$, the decomposition of either
    $\so(2p,2q+1)$ or $\so(2p+1,2q)$ into irreducible $\su(p,q)$-submodules is given by the sum
    \[
        \su(p,q) \oplus \R \oplus (\wedge^2\C^n)_\R \oplus \C^{p,q}_\R.
    \]
    For $p=q=2$ the same isomorphism holds, but the second to last summand is not irreducible.
    Furthermore, if we induce on the above sum a corresponding Lie algebra structure, then we have
    \begin{align*}
        [\su(p,q) \oplus \R \oplus (\wedge^2\C^n)_\R, \C^{p,q}_\R] &= \C^{p,q}_\R \\
        [\C^{p,q}_\R, \C^{p,q}_\R] &= \su(p,q) \oplus \R \oplus (\wedge^2\C^n)_\R \\
        [\R, \C^{p,q}_\R] &= \C^{p,q}_\R
    \end{align*}
\end{lemma}
\begin{proof}
    We observe that the embedding of $\uni(p,q)$ into either $\so(2p,2q+1)$ or $\so(2p+1,2q)$ is the composition of the embedding of $\uni(p,q)$ into $\so(2p,2q)$ and that of the latter into either $\so(2p,2q+1)$ or $\so(2p+1,2q)$, respectively. Both embeddings define symmetric pairs, so we can use their structure and Lemma~\ref{lem:su-so-decomp-phi} to obtain the conclusion. For that we also use the isomorphism $\C^{p,q}_\R \simeq  \R^{2p,2q}$ of $\so(2p,2q)$-modules.
\end{proof}

We recall the following property of spaces with scalar product.

\begin{lemma}
    \label{lem:so(E)}
    Let $E$ be a finite dimensional real vector space with scalar product. Then, the assignment
    \[
        u\wedge v \mapsto \left<\cdot,u\right> v - \left<\cdot,v\right> u.
    \]
    defines an isomorphism $\varphi : \wedge^2 E \rightarrow \so(E)$ of $\so(E)$-modules.
\end{lemma}

As a consequence, we obtain the following result.

\begin{lemma}
    \label{lem:wedgeCpqR}
    Let $p, q \geq 1$ and $n = p+q \geq 3$. For $(p,q) \not= (2,2)$, the decomposition of $\wedge^2 \C^{p,q}_\R$ into irreducible $\su(p,q)$-submodules is given by the sum
    \[
        \su(p,q) \oplus \R \oplus (\wedge^2\C^n)_\R.
    \]
    For $p=q=2$ the same decomposition holds, but the last summand is not irreducible.
\end{lemma}
\begin{proof}
    Consider $\C^{p,q}_\R$ with the (unique up to a constant) $\su(p,q)$-invariant scalar product from Lemma~\ref{lem:Cpq-scalar}. By Lemma~\ref{lem:so(E)} there is a natural isomorphism $\varphi : \wedge^2 \C^{p,q}_\R \rightarrow \so(\C^{p,q}_\R)$ of $\so(\C^{p,q}_\R)$-modules. We also note that $\so(\C^{p,q}_\R)$ is isomorphic to $\so(2p,2q)$ so that the embedding of $\su(p,q) \hookrightarrow \so(2p,2q)$ given by Equation~\eqref{eq:su-so1} is an intertwining map for the representations of $\su(p,q)$ on $\wedge^2\C^{p,q}_\R$ and of $\so(\C^{p,q}_\R)$ on itself. In particular, $\varphi$ is an isomorphism of $\su(p,q)$-modules and the result follows from Lemma~\ref{lem:su-so-decomp-phi}.
\end{proof}

\end{document}